\documentclass[11pt]{article}
\usepackage{fullpage}

\usepackage{amsthm}
\usepackage{amsmath, amssymb, graphicx, url}
\usepackage{amsmath,amsfonts,amssymb,bm} 

\DeclareFontFamily{U}{mathx}{\hyphenchar\font45}
\DeclareFontShape{U}{mathx}{m}{n}{
      <5> <6> <7> <8> <9> <10>
      <10.95> <12> <14.4> <17.28> <20.74> <24.88>
      mathx10
      }{}
\DeclareSymbolFont{mathx}{U}{mathx}{m}{n}
\DeclareFontSubstitution{U}{mathx}{m}{n}
\DeclareMathAccent{\widecheck}      {0}{mathx}{"71}

\newcommand{\Tsch}{{Cheby\v{s}ev}} 
\newcommand{\bmc}{\boldsymbol{c}}

\newcommand{\bse}{{\boldsymbol{e}}}

\newcommand{\bsj}{{\boldsymbol{j}}}

\newcommand{\bsv}{{\boldsymbol{v}}}
\newcommand{\bsw}{{\boldsymbol{w}}}
\newcommand{\bsx}{{\boldsymbol{x}}}
\newcommand{\bsy}{{\boldsymbol{y}}}

\newcommand{\bsQ}{{\boldsymbol{Q}}}

\newcommand{\bsmu}{{\boldsymbol{\mu}}}
\newcommand{\bsnu}{{\boldsymbol{\nu}}}

\newcommand{\bsrho}{{\boldsymbol{\rho}}}

\newcommand{\bsPsi}{{\boldsymbol{\Psi}}}

\newcommand{\bbC}{{\mathbb{C}}}

\newcommand{\bbN}{{\mathbb{N}}}

\newcommand{\bbP}{{\mathbb{P}}}

\newcommand{\bbR}{{\mathbb{R}}}

\newcommand{\bbT}{{\mathbb{T}}}

\newcommand{\C}{{\mathbb{C}}} %
\newcommand{\N}{{\mathbb{N}}} %
\newcommand{\R}{{\mathbb{R}}} %

\DeclareSymbolFont{bbold}{U}{bbold}{m}{n}
\DeclareSymbolFontAlphabet{\mathbbold}{bbold}

\newcommand{\calD}{{\mathcal{D}}}

\newcommand{\calG}{{\mathcal{G}}}

\DeclareSymbolFont{bbold}{U}{bbold}{m}{n}
\DeclareSymbolFontAlphabet{\mathbbold}{bbold}
\newcommand{\domain}{{\mathrm D}} %
\newcommand{\ii}{{\mathrm i}} %

\newcommand{\eps}{{\varepsilon}}
\newcommand{\be}{\begin{equation}}
\newcommand{\ee}{\end{equation}}
\usepackage{microtype} %

\usepackage[colorlinks=true,linkcolor=black,citecolor=black,urlcolor=black]{hyperref}
\usepackage{bookmark}
\makeatletter %
  \providecommand*{\toclevel@author}{999}
  \providecommand*{\toclevel@title}{0}
\makeatother

\newcommand{\cX}{\mathcal{X}}
\newcommand{\cY}{\mathcal{Y}}
\newcommand{\cA}{\mathcal{A}}

\newcommand{\cH}{\mathcal{H}}
\newcommand{\cG}{\mathcal{G}}
\newcommand{\cF}{\mathcal{F}}
\newcommand{\cE}{\mathcal{E}}
\newcommand{\cD}{\mathcal{D}}

\newcommand{\rF}{\mathrm {F}}
\newcommand{\rS}{\mathrm {S}}

\newcommand{\cM}{{\mathcal M}} %
\newcommand{\bG}{{\boldsymbol G}}

\newcommand{\norm}[2][]{\|#2\|_{#1}}
\newcommand{\normlr}[2][]{\left\|#2\right\|_{#1}}
\newcommand{\set}[2]{\{#1 : #2\}}
\newcommand{\setlr}[2]{\left\{#1 : #2\right\}}
\newcommand{\dup}[2]{\langle #1,#2 \rangle}
\newcommand{\duplr}[2]{\left\langle #1,#2 \right\rangle}
\DeclareMathOperator{\supp}{supp}
\DeclareMathOperator*{\essinf}{ess\,inf}
\newcommand{\dd}{\;\mathrm{d}}
\usepackage{enumitem}

\newtheorem{theorem}{Theorem}[section]
\newtheorem*{acknowledgement}{Acknowledgment}

\newtheorem{corollary}[theorem]{Corollary}

\newtheorem{definition}[theorem]{Definition}
\newtheorem{example}[theorem]{Example}

\newtheorem{lemma}[theorem]{Lemma}

\newtheorem{proposition}[theorem]{Proposition}
\newtheorem{remark}[theorem]{Remark}

\newtheorem{assumption}[theorem]{Assumption}

\usepackage{authblk}

\begin{document}
\title{Neural and spectral operator surrogates:\\ unified construction and expression rate bounds}

\author[1]{Lukas Herrmann}
\author[2]{Christoph Schwab}
\author[3]{Jakob Zech}

\affil[1]{\small Johann Radon Institute for Computational and Applied Mathematics, Austrian Academy of Sciences, Altenbergerstrasse 69, 4040 Linz, Austria
  
  \texttt{lukas.herrmann@alumni.ethz.ch}}
\affil[2]{\small Seminar for Applied Mathematics, ETH Z\"urich, R\"amistrasse 101, 8092 Zurich, Switzerland
  
  \texttt{christoph.schwab@sam.math.ethz.ch}}
\affil[3]{\small Interdisziplin\"ares Zentrum f\"ur wissenschaftliches Rechnen, Universit\"at Heidelberg, Im Neuenheimer Feld 205, 69120 Heidelberg, Germany
  
  \texttt{jakob.zech@uni-heidelberg.de}}

\maketitle

\abstract{ Approximation rates are analyzed for deep surrogates of
  maps between infinite-dimensional function spaces, arising e.g. as
  data-to-solution maps of linear and nonlinear partial differential
  equations.  Specifically, we study approximation rates for
  \emph{Deep Neural Operator} and \emph{Generalized Polynomial Chaos
    (gpc) Operator} surrogates for nonlinear, holomorphic maps between
  infinite-dimensional, separable Hilbert spaces.  Operator in- and
  outputs from function spaces are assumed to be parametrized by
  stable, affine representation systems.  Admissible representation
  systems comprise orthonormal bases, Riesz bases or suitable tight
  frames of the spaces under consideration.  Algebraic expression rate
  bounds are established for both, deep neural and {spectral} operator
  surrogates acting in scales of separable Hilbert spaces containing
  domain and range of the map to be expressed, with finite Sobolev or
  Besov regularity.
  We illustrate the abstract concepts by expression rate bounds for
  the coefficient-to-solution map for a linear elliptic PDE on the torus.
}
\\

\section{Introduction}
\label{sec:intro}
In recent years, deep learning (DL) based numerical methods have started to
impact the numerical solution of (parametric) partial
differential equations (PDEs) at every stage of the solution
process. Deep Neural Networks (DNNs) have been promoted as efficient
approximation architectures for PDE solutions and parametric PDE
  response manifolds. 
However, the theoretical understanding of
the methodology remains underdeveloped; 
quoting
\cite[Sec.~4.1.1]{AdlOktem17}: ``applying deep learning to infinite
dimensional spaces is associated with a number of fundamental
questions regarding convergence..., if it converges, in what sense?''
\subsection{Existing results}
\label{sec:PrevRes}
Recent successful examples of deployment of DL surrogates in numerical
PDE solution algorithms
(e.g. \cite{KarnPINN_JCP2019,TripBiliJCP2018,KarnBIPDNN_JCP2019,HZE19}),
has promoted \emph{neural network approximation architectures for
  PDE solution approximation}.  
The related question of approximation
rates of NN based PDE discretizations has been answered in a number of 
settings.  
We mention only
\cite{OPS20_2738,MOPS22_2877,Xu_2020,LODSZ22_991,BGKP19,berner2021modern}
and the references there.  
These approximation rate results were
developed in function spaces on finite-dimensional domains.

A more recent and distinct development addresses %
\emph{neural networks for operator learning}, i.e. for the neural
network emulation of maps between infinite dimensional function
spaces, such as solution operators of PDEs, coefficient-to-solution
and shape-to-solution maps for elliptic PDEs, to name but a few.
These have been promulgated under the acronym ``Neural operators'',
``O-Nets'', ``operator learning'', see, e.g.,
\cite{lu2021comprehensive,li2021physicsinformed} and the references
there.  
Versions of the \emph{DNN universal approximation theorems for
  operator learning} have been established, following the pioneering
work \cite{ChenChen1993}, recently in
\cite{yu2021arbitrarydepth,deng2021convergence,li2021physicsinformed,lanthaler2021error,acciaio2022metric}.
In these references, generic approximation properties of operators for
several architectures of DNNs of in principle arbitrary large depth
and width have been established.  These results are analogous to the
early, universal approximation results of DNNs for function
approximation from the 90s of the previous century, as
e.g., \cite{hornik1990universal} and \cite{pinkus1999approximation} and
the references there.

Contrary to the mentioned universality
theorems, proofs of \emph{operator expression rate bounds} tend to be
problem-specific, including assumptions on regularity of input and
output data of the operators of interest, and some structural
assumptions on the operator mappings.  With domains and/or ranges of
the operators of interest being infinite-dimensional, as a rule,
overcoming the \emph{curse of dimensionality} (CoD) in the proofs of
operator emulation bounds is necessary.  
The results in the present paper leverage progress in recent years
on approximation rate bounds for gpc representations
of such maps for DNN approximation.  This line of research was
initiated in \cite{SchwabZechAA}, building on earlier results on gpc
emulation rates in \cite{BCDS17_2452} and the references there.
{We also mention recent publications where operator expression rate bounds
    have been proved for particular en- and decoding systems, such as KL-expansions
\cite{LanthPCA}.
}

For coefficient-to-solution maps of elliptic PDEs, \emph{on domains
  consisting of smooth coefficient functions}, it was shown recently
in \cite{MS21_984} that \emph{exponential expression rates of solution
  operators} is possible with certain deep ReLU NNs.  This result
leverages the exponential encoding and decoding of smooth (analytic)
functions with tensorized polynomials by spectral collocation,
combined with a ReLU NN which emulates a Gaussian Elimination Method
for regular matrices of size $N$ in NN size $O(N^4)$, and scaling
polylogarithmically in terms of the target accuracy $\eps$ of the
solution vector.
\subsection{Contributions}
\label{sec:Contr}
We establish expression rate bounds for DNN emulations of holomorphic
maps between subsets of (scales of) infinite-dimensional spaces $\cX$
and $\cY$.  We focus in particular on maps between scales
$\{\cX^s\}_{s\geq 0}$ and $\{\cY^t\}_{t\geq 0}$ of spaces of finite
regularity, such as function spaces of Sobolev- or Besov-type on
``physical domains'', being for example bounded subsets $\domain$ of
Euclidean space.  Mappings between function spaces on manifolds $\cM$
are also covered by the present operator expression rate bounds, upon
introduction of stable bases in suitable function spaces on $\cM$, or
also on space-time cylinders $\cM \times [0,T]$.

A typical example is the (nonlinear) coefficient-to-solution map for
linear, elliptic or parabolic partial differential equations of second
order which we develop for illustration in some detail in
Section~\ref{sec:Expl} ahead.  
In this ``linear elliptic PDE'' case, 
the (nonlinear) coefficient-to-solution map is holomorphic between
suitable subsets of $L^\infty(\domain)$ (accounting for positivity)
and $H^1(\domain)$ (accounting for homogeneous essential boundary
conditions on $\partial\domain$).  The notion of operator holomorphy
requires complexifications of the domain and range spaces, in the
(common) case that physical modelling will initially comprise only
spaces of real-valued functions in $\domain$.  To simplify
technicalities of exposition, we develop the theory for separable,
real Hilbert spaces $\cX$. 
Complexificiation is then a standard
process resulting in a (canonical) extension \cite{MunozCplxB}.

Our results show that for data with finite Sobolev- or Besov
regularity, there exist operator surrogates of either deep neural
network or of generalized polynomial chaos type such that
approximation rates afforded by linear approximation schemes are
essentially preserved by the surrogate operators.  This generalizes
the recent result in \cite{MS21_984} where analyticity of inputs was
exploited in an essential fashion to the more realistic, finite
regularity setting, in rather general classes of function spaces.  In
addition, our \emph{proofs of these results are constructive} allowing
for a \emph{deterministic construction of the surrogate maps with a
  set budget of pre-defined, numerical operator queries}.  Our main
results, Theorems~\ref{thm:main} and~\ref{thm:main2}, ensure
\emph{worst case and mean square generalization error bounds for 
neural operator surrogates} thus computed.  
The algebraic operator expression rates are 
limited by the approximation rates of the
encoding and decoding operators entering into the construction of the
surrogates.  
Theorem~\ref{thm:int} then has corresponding results for the 
{spectral} operator surrogate.

We note that alternative approaches to analyzing the generalization
error of operator surrogates, such as methods from statistical
learning theory (e.g. \cite{liu2022deep,lanthaler2023error} and the references there),
deliver lower approximation rates (these results do not require
holomorphy of $\cG$, however).
\subsection{Layout of this text}
\label{sec:Layout}
In Section~\ref{sec:Setting}, we present an abstract function space
setting, in which the operators and their surrogates will be analyzed. 
We precise in particular the notion of stable bases in
smoothness scales via isomorphisms to sequence spaces, 
comprising orthonormal bases in
separable Hilbert spaces of 
Fourier- \cite{li2020fourier}, \Tsch- \cite{FanasOsele2022} and
Karhunen-Lo\`{e}ve type,
as well as biorthogonal bases of spline and wavelet type 
(see, e.g., \cite{WavNeurOp2022} and the references there). 
Examples are furnished by Sobolev and Besov spaces, %
and by reproducing kernel Hilbert spaces of covariance operators in statistics 
(e.g.\ \cite{TriebelWavelets2008} and \cite{SteinwartScovel12} 
 and the references there).

Section~\ref{sec:MainReS} gives a succinct statement of our main
results in Theorems \ref{thm:main} and \ref{thm:main2}; these theorems
provide %
expression rate bounds of Deep Neural Network operator emulations and
of generalized polynomial chaos emulations for holomorphic maps
between separable
Hilbert spaces admitting \emph{stable, biorthogonal bases}.  
There, a
key role in building appropriate encoding and decoding operators for
neural operator networks is taken by \emph{dual bases}, which must, to
some extent, be available explicitly in order to construct the
encoders and decoders.  Section~\ref{sec:proof} provides the proof of
Theorem~\ref{thm:main} which asserts the \emph{existence} of operator
emulations which preserve algebraic encoding and decoding error bounds
of input and output data.

In Section~\ref{sec:int}, we discuss in further detail the second,
novel class of operator emulations dubbed \emph{spectral} or
\emph{generalized polynomial chaos operator (``gpc operator'', {or ``spectral operator''}) surrogates}. 
Its \emph{deterministic construction} 
is via finitely truncated gpc expansions of holomorphic
parametric maps, resulting from suitable encoders and decoders in the
domain and range, respectively.  
Sparse gpc operator surrogates
provide a construction (via ``stochastic collocation'') for operator
emulations, %
i.e., the proof of Theorem~\ref{thm:int} on gpc operator expression
rates yields an explicit deterministic construction procedure
realizing the proposed operator emulations.

Finally, in Section~\ref{sec:Expl} we illustrate the abstract theory with
an example.
	\subsection{Notation}
	\label{sec:Not}
        We write $\N = \{1,2,...\}$ and $\N_0 = \{0,1,2,...\}$. 
        Throughout, $C \lesssim D$ means that $C$
          can be bounded by a multiple of $D$, independently of
          parameters which $C$ and $D$ may depend on.
          $C \gtrsim D$ is defined as $D \lesssim C$, and 
          $C\simeq D$ as $C\lesssim D$ and $C \gtrsim D$.

          For a {real} Hilbert space $\cX$, 
          the inner product of $v$ and $w\in\cX$ is denoted by $\dup{v}{w}$.  
          The space of real-valued, square {summable} sequences indexed over $\N$ 
          is denoted by $\ell^2(\N)$. 
          Complex valued, square summable sequences shall be denoted with $\ell^2(\N,\C)$. {The (unique, cf. e.g. \cite{MunozCplxB}) complexification of a real Hilbert space $\cH$ is denoted by $\cH_\C$.}
	\section{Setting}
	\label{sec:Setting}

	We fix notation and introduce, following established practice
	in statistical learning theory, encoder and decoder operators 
	as stipulated in \cite{ChenChen1993,lanthaler2021error}.
        Throughout, $\cX$, $\cY$ {shall denote} separable Hilbert spaces over $\R$.

	\subsection{Framework}
	\label{sec:AbsFrmWk}
        ``Operator learning'' refers, in the present paper, to
        procedures of emulation of (not necessarily linear) maps
        $\cG:\cX\to\cY$.
        We will address \emph{existence and error
          bounds of surrogate maps $\tilde{\cG}$}, subject to a finite
        number $N$ of parameters defining $\tilde{\cG}$. 
        If we wish to
        emphasize the dependence on $N$, we also write
        $\tilde\cG_N=\tilde\cG$.
        
        On suitable %
        subsets $S\subseteq\cX$ of admissible
        input data, we will consider convergence rates in $N$ either
        in terms of the \emph{worst case error}
	\be\label{eq:G-tG} \sup_{a\in S} \| \cG(a) - \tilde{\cG}_N(a)
        \|_{\cY} \ee or the \emph{mean square error}
        \be\label{eq:G-tG2} \left( \int_{a \in S} \| \cG(a) -
          \tilde{\cG}_N(a) \|_{\cY}^2 \dd\zeta(a) \right)^{1/2}, \ee for a
        measure $\zeta$ on $S$ equipped with the Borel sigma algebra.
        
	As in \cite{li2021physicsinformed,deng2021convergence} and the
        references there, we seek surrogates $\tilde{\cG}_N$
        of the form
	\be\label{eq:GenForm} 
        \tilde{\cG}_N := \cD \circ \tilde{\bG}_N \circ \cE, 
        \ee
	where $\cE: \cX\to \ell^2(\bbN)$ and $\cD: \ell^2(\bbN) \to
        \cY$ denote the so-called \emph{encoder} and \emph{decoder}
        maps.  The encoder allows to express elements in $\cX$ in a
        certain (efficient) representation system.  For example, $\cE$
        could map vectors in $\cX$ to their Fourier coefficients
        w.r.t.\ some fixed orthonormal basis in $\cX$, and $\cD$ could
        perform the opposite operation w.r.t.\ another fixed
        orthonormal basis in $\cY$. While we restrict ourselves to
        linear encoders/decoders, we will give a more general
        framework and further details in the next subsections.  
        The parametric approximations
        $\tilde{\bG}_N : \ell^2(\bbN)\to \ell^2(\bbN)$ in
        \eqref{eq:GenForm} belong to hypothesis classes comprising
        $N$-term polynomial chaos expansions or deep neural networks
        depending on $N$ parameters.
\subsection{Representation systems}
\label{sec:BiOrth}
We discuss several representation systems and corresponding pairs
  of encoder/decoder maps. Specifically, we will admit (possibly
  redundant) biorthogonal systems such as Riesz bases realized by
  Multiresolution Analyses (MRAs) or by 
  {Finite Element frames (e.g.\ \cite{HSS08}) or 
  more general directional representation systems \cite{GK14} 
  comprising wavelet, shearlet (e.g.\ \cite{PhGShearlFr11}) and curvelet frames,}
  which offer additional flexibility over mere orthogonal bases. This framework
  also comprises Karhunen-Lo\`eve eigenfunctions, Fourier bases and
  other, fully orthogonal families as particular cases. 
  See e.g. \cite[Chap. 2,5,6.3]{TriebelWavelets2008} for 
  general discussion and constructions.

  \subsubsection{Frames}
  \label{sec:Frames}
  Constructions of concrete representation systems are often
  simplified when one insists on stability, but the basis property is
  relaxed.  This leads to the concept of \emph{frames}, which we now
  shortly recall.  It comprises biorthogonal wavelet bases as a
  particular case, and allows in particular also iterative realization
  on unstructured simplicial partitions of polyhedra via the so-called
  BPX multi-level iteration {(see, e.g., \cite{HSS08}, and the orginal
  construction due to P.\ Oswald in space dimension $d=2$ \cite{OswaldBPXP12d},
  and subsequently in polyhedra, in \cite{OswaldML94} and the references
  there. A frame property in $H^1_0(\domain)$ 
  of the subspace splittings furnished by the BPX iteration is also implicitly
  shown in \cite{BPWX}).}

  \begin{definition}
  A collection $\bsPsi = \set{\psi_j}{j\in\N}\subset
  \cX$ %
  is called a \emph{frame for $\cX$}, if the \emph{analysis operator}
$$
\rF: \cX \to \ell^2(\N): v\mapsto (\dup{v}{\psi_j})_{j\in\N}
$$
is boundedly invertible between $\cX$ and ${\rm range}(\rF)\subset \ell^2(\N)$.
\end{definition}
The adjoint $\rF'$ of the analysis operator is called the
\emph{synthesis operator}. 
It is given by
\begin{equation}\label{eq:calFp}
\rF': \ell^2(\N) \to \cX: \bsv \mapsto \bsv^\top\bsPsi \;.
\end{equation}
The \emph{numerical stability of frames} 
is quantified by the \emph{frame bounds}
\be \label{eq:FrBd} 
\lambda_{\bsPsi} := \inf_{0\ne v \in \cX} \frac{\| \rF v \|_{\ell^2}}{\| v \|_\cX} \;,
\quad
\Lambda_{\bsPsi} := \| \rF \|_{\cX\to \ell^2} = \sup_{0\ne v \in \cX}  \frac{\| \rF v \|_{\ell^2}}{\| v \|_\cX} \;.
\ee
{
\emph{Parseval frames} are frames $\bsPsi$ with ideal conditioning 
$\lambda_{\bsPsi} = \Lambda_{\bsPsi} = 1$.}
\begin{remark}\label{rmk:FrBd}
    Since $\|\rF'\|_{\ell^2\to\cX}=\|\rF\|_{\cX\to\ell^2}$, 
    \eqref{eq:FrBd} implies that for all $\bsnu\in\ell^2(\N)$
    \begin{equation}\label{eq:rmk:FrBd}
      \normlr[\cX]{\sum_{j\in\N}\nu_j\psi_j}^2
      =\norm[\cX]{\rF'\bsnu}^2
      \le \Lambda_{\bsPsi}^2\sum_{j\in\N}\nu_j^2 
       = \Lambda_{\bsPsi}^2 \norm[\ell_2]{\bsv}^2.
    \end{equation}
\end{remark}

The \emph{frame operator} $\rS := \rF'\rF: \cX \to \cX$ is 
boundedly invertible, self-adjoint and positive
(e.g. \cite{Christensen,Heil}) 
with
$\| \rF' \rF \|_{\cX\to \cX} = \Lambda_{\bsPsi}^2$ 
and
$\| (\rF' \rF)^{-1} \|_{\cX\to \cX} = \lambda_{\bsPsi}^{-2}$,
\cite[Lemma 5.1.5]{Christensen}.

{
With the pseudoinverse $(\rF')^\dagger = \rF(\rF'\rF)^{-1}$ 
of the synthesis operator, given by 
$$
(\rF')^\dagger: \cX \to \ell^2(\N): f \mapsto \{ \langle f, S^{-1} \psi_j \rangle \}_{j\in \N},
$$
the \emph{frame decomposition theorem} asserts that
every $f\in \cX$ can be uniquely and stably reconstructed
from a corresponding sequence of frame coefficients via
$$
f = \rF'(\rF')^\dagger f 
= \sum_{j\in \N} \langle f,S^{-1}\psi_j \rangle \psi_j
= \sum_{j\in \N} \langle f,\psi_j\rangle S^{-1} \psi_j 
\;.
$$
This highlights the importance of
}
the collection
$\tilde{\bsPsi}:= \rS^{-1}\bsPsi$. It
is a frame for $\cX$ which is referred to 
as the \emph{canonical dual frame} of $\bsPsi$.  
Its analysis operator is
$\tilde{\rF} := {(\rF')^\dagger = } \rF (\rF'\rF)^{-1}$,
and its frame bounds
\eqref{eq:FrBd} are $\lambda_{\bsPsi}^{-1}$ and
$\Lambda_{\bsPsi}^{-1}$, respectively.

{The frame decomposition theorem takes the form (e.g.} \cite{Christensen,Heil})
\begin{equation}\label{eq:compid}
  \rF' \tilde{\rF} = I\qquad\text{on}\qquad \cX.
\end{equation}
Whence every $v\in \cX$ has a representation 
$v = \bsv^\top\bsPsi$ with $\bsv = \tilde\rF(v)\in \ell^2(\N)$, 
and
\be\label{eq:FrStab} 
\Lambda_\bsPsi^{-1} \leq \frac{\| \bsv
\|_{\ell^2}}{\|v\|_\cX} \leq \lambda_{\bsPsi}^{-1} \;.  
\ee
Property \eqref{eq:FrStab} is equivalent to $\bsPsi$ 
being a frame for $\cX$ (see, e.g., \cite[Thm.~8.29 (b)]{Heil}).  

{
\begin{remark}\label{rmk:CDE} [Continuous-discrete equivalence]
Given a (generally nonlinear) map $\cG:\cX\to\cY$ 
between separable Hilbert spaces $\cX$ and $\cY$ over $\R$
which are endowed with frames $\bsPsi_\cX$ and $\bsPsi_\cY$, respectively,
there exists a \emph{coordinate map} 
$\bG = \bG(\bsPsi_\cX, \bsPsi_\cY):\ell_2(\N) \to \ell_2(\N)$ 
such that there holds %
\begin{equation}\label{eq:CDE}
\cG 
= \rF'_\cY \circ \bG \circ (\rF'_\cX)^\dagger 
= \rF'_\cY \circ \bG \circ \tilde{\rF}_\cX
\;.
\end{equation}
Here and throughout, 
in case that the representation systems $\bsPsi_\cX$ and $\bsPsi_\cY$ 
are clear from the context, 
we write $\bG$ in place of $\bG(\bsPsi_\cX, \bsPsi_\cY)$.
\end{remark}
}
\begin{remark} \label{rmk:FramKer}
[Uniqueness of coordinate sequence]
Unless
$\bsPsi$ is a Riesz basis (see below), 
representation of $v\in \cX$ as
$v = \bsv^\top\bsPsi$ is generally not unique: 
there holds
$\ell^2(\N) = {\rm ran}(\rF) \oplus^\perp {\rm ker}(\rF')$ 
and
$\bsQ := \tilde{\rF} \rF'$ is the orthoprojector onto
${\rm ran}(\rF)$.  
We refer to \cite{Christensen,Heil} and the references there.
\end{remark}
\subsubsection{Riesz Bases}
\label{sec:RieszB}
Riesz bases are special cases of frames
which are defined as follows {(e.g. \cite{Christensen,Heil}).}
\begin{definition}\label{def:RieszB}
  A sequence $\bsPsi = \{ \psi_j \}_{j\in \N}\subset \cX$ is a
  Riesz basis of $\cX$ if there exists a bounded bijective operator
  $A:\cX\to\cX$ and an orthonormal basis $(e_j)_{j\in\N}$ such that
  $\psi_j=Ae_j$ for all $j\in\N$.
\end{definition}

Every Riesz basis of $\cX$ is a frame:
there exist \emph{Riesz constants}
$0< \lambda_\bsPsi \leq \Lambda_\bsPsi <\infty$ 
such that
\begin{equation}\label{eq:StabRB}
\forall (c_j)_{j\in\N}\in\ell^2(\N): 
\quad 
{\lambda_\bsPsi^2} \sum_{j\in\N} | c_j |^2 
\leq 
\left\|\sum_{j\in\N} c_j \psi_j \right\|_{\cX}^2 
\leq 
{\Lambda_\bsPsi^2} \sum_{j\in\N} | c_j |^2 
\;.
\end{equation}
{
There holds a continuous-discrete equivalence:
any $v\in \cX$ can be \emph{equivalently represented}
by the sequence 
$\bmc = (c_j)_{j\in \N}$ of its coefficients w.r. to $\bsPsi$.
}

The canonical dual frame 
$\tilde\bsPsi = \{\tilde\psi_j\}_{j\in\N}$ of $\bsPsi$ 
is also a Riesz basis of $\cX$, and is referred to as the
the \emph{dual basis} or the \emph{biorthogonal system} to $\bsPsi$, 
since for all $j$, and all $k\in\N$ holds $\dup{\psi_j}{\tilde\psi_k}=\delta_{kj}$. 
We refer to \cite[Sec.~5]{Christensen} for further details and proofs.
\begin{remark}\label{rmk:RBConstr}
Constructions of piecewise polynomial Riesz bases for
Sobolev spaces in polytopal domains $\domain\subset \bbR^d$
are available
(e.g. \cite{DavydovStevenson2006,RobSurvey09} and the references there). 
{
We mention in particular \cite{DahmStevEBE} where 
locally supported wavelet bases for $C^0$-Lagrange finite element spaces on regular, 
simplicial triangulations of polytopal domains $\domain$ are constructed.
The constructions in \cite{DahmStevEBE} provide 
Riesz bases in $\cX = H^1(\domain)$ and $\cX = H^1_0(\domain)$.
}
\end{remark}
\subsubsection{Orthonormal Bases}
\label{sec:ONB}
Orthonormal bases (ONBs) are particular instances of frames and Riesz
bases: 
if $\bsPsi$ is an orthonormal basis of $\cX$, then $\tilde\bsPsi=\bsPsi$. 
This includes, for example, 
Fourier-bases \cite{Heil},
Daubechies - type wavelets \cite{DaubONWav} and 
orthogonal polynomials \cite{Szego3rdOrthPol}.
It also includes orthonormal bases obtained
by principal component analyses associated with a covariance operator
corresponding to a Gaussian measure on $\cX$ as commonly used in
statistical learning theory (e.g. \cite{SteinwartScovel12}).  
Such bases are generally not explicitly available, 
but may be approximately calculated in practice.
\begin{example}\label{ex:Fourier}
  Denote by $\bbT^d$ the $d$-dimensional torus.
  The Fourier basis $\bsPsi$ is an ONB
  of $\cX=L^2(\bbT^d)$. The analysis
  and synthesis operators $\rF$, $\rF'$ are in
  this case the Fourier transform and its inverse transform.
\end{example}
\subsection{Encoder and decoder}
\label{sec:EncDec}
In the following we use the notation
  \begin{equation}\label{eq:psieta}
    \bsPsi_\cX=(\psi_j)_{j\in\N},\quad
    \tilde\bsPsi_\cX=(\tilde\psi_j)_{j\in\N},\quad
    \bsPsi_\cY=(\eta_j)_{j\in\N},\quad
    \tilde\bsPsi_\cY=(\tilde\eta_j)_{j\in\N}.
  \end{equation}
to denote frames and their {canonical} dual frames of $\cX$, $\cY$ respectively.
With the corresponding analysis operators
$\rF_\cX$, $\rF_\cY$ %
the encoder/decoder pair in \eqref{eq:GenForm} %
is defined 
by the analysis and synthesis operators 
which are given by {(cf.\ \eqref{eq:CDE})}
\begin{equation}\label{eq:ED}
  \cE:={\tilde\rF}_\cX=\begin{cases}
    \cX\to \ell^2(\N)\\
    x\mapsto (\dup{x}{\tilde\psi_j})_{j\in\N},
  \end{cases}
  \qquad
  \cD:=\rF_\cY'=\begin{cases}
    \ell^2(\N)\to \cY\\
    (y_j)_{j\in\N}\mapsto \sum_{j\in\N}y_j\eta_j.
  \end{cases}  
\end{equation}
\begin{remark}
  If
  $\bsPsi_\cX$, $\bsPsi_\cY$ are Riesz bases of $\cX$, $\cY$,
  respectively, then the encoder $\cE : \cX\to\ell^2(\N)$ and
    decoder $\cD :\ell^2(\N)\to\cY$ in \eqref{eq:ED} are boundedly
  invertible operators.
\end{remark}

\begin{remark}\label{rmk:Thresh}
  Encoders and decoders with rate-optimal performance for subsets
  $\cX^s\subset \cX$ are obtained from \emph{$n$-term truncation}.  In
  the most straightforward case, \emph{linear $n$-term truncation} of
  representations $u\in \cX$ will ensure rate-optimal approximations
  for $\cX^s$ being classical Sobolev or Besov spaces with summability
  index $p\geq 2$.  It is well-known that MRAs which constitute Riesz
  bases in $\cX$ afford \emph{nonlinear encoding by coefficient
    thresholding}.  This could also be referred to as \emph{adaptive
    encoding}.  Such encoders are known to ensure rate-optimal
  approximations for a given budget of $n$ coefficients for
  considerably larger set $\cX^s \subset \cX$, comprising in
  particular Besov spaces in {bounded domains $\domain\subset \R^d$} 
  with summability indices $q\in (0,1]$ 
  (see, e.g., \cite{TriebelWavelets2008,TriebelBases2010}).
\end{remark}
\subsection{Smoothness scales}
\label{sec:Scales}
Our analysis will require subspaces of $\cX$ and $\cY$
exhibiting ``extra smoothness''. 
Typical instances are Sobolev and
Besov spaces with ``$s$-th weak derivatives bounded''.  
It is well-known, that membership in such function classes 
can be encoded via weighted summability of expansion
coefficients.  

To formalize this, for a fixed strictly positive, monotonically
decreasing {weight} sequence $\bsw=(w_j)_{j\in\N}$ such that
$\bsw^{1+\epsilon}\in\ell^{1}(\N)$ for all %
$\epsilon>0$, we introduce scales of Hilbert spaces
$\cX^s\subseteq \cX$, $\cY^t \subseteq \cY$ for $s, t\ge 0$ with
norms\footnote{All of the following remains valid if we use distinct
weight sequences $(w_{\cX,j})_{j\in\N}$ and $(w_{\cY,j})_{j\in\N}$
to define $\cX^s$, $\cY^t$ respectively.  
We refrain from doing so for simplicity of presentation.}
\begin{equation}\label{eq:XsYt}
  \norm[\cX^s]{x}^2:=\sum_{j\in\N}\dup{x}{\tilde{\psi}_{j}}^2 w_{j}^{-2s},
  \qquad
  \norm[\cY^t]{y}^2:=\sum_{j\in\N}\dup{y}{\tilde{\eta}_{j}}^2 w_{j}^{-2t}.
\end{equation}

  \begin{lemma} \label{lem:cXs}
    Let $s\ge 0$. 
    The space
    $\cX^s=\set{x\in\cX}{\norm[\cX^s]{x}<\infty}$ is a Hilbert space
    with inner product
    $\dup{x}{x'}_{\cX^s}=\sum_{j\in\N}\dup{x}{\tilde\psi_j}\dup{x'}{\tilde\psi_j}w_j^{-2s}$.
  \end{lemma}
  \begin{proof}
    Clearly $\dup{\cdot}{\cdot}_{\cX^s}$ defines an inner product on
    the set $\cX^s$ compatible with the norm $\norm[\cX^s]{\cdot}$. We
    need to show that $\cX^s$ is closed w.r.t.\ this norm.

    Denote $\cE=\tilde\rF_\cX$, $\cD=\rF_\cX'$ and recall that
    $\cE(\cX)$ is a closed subspace of $\ell^2(\N)$ due to the
    property $\norm[\ell^2(\N)]{\cE(x)} \ge \lambda_\cX\norm[\cX]{x}$.
    Furthermore, denote in the following by
    $\ell_s^2(\N)$ the sequence space of $\bsx\in\ell^2(\N)$ such that
    $\norm[\ell^2_s]{\bsx}^2:=\sum_{j\in\N}x_j^2w_j^{-2s}<\infty$.
    Note that $\ell^2_s(\N)$ is closed, %
    and
    $\norm[\ell^2_s]{\cE(x)}=\norm[\cX^s]{x}$.

    Take a Cauchy sequence $(x_n)_{n\in\N}\subseteq\cX^s$ w.r.t.\
    $\norm[\cX^s]{\cdot}$. Then
    $(\cE(x_n))_{n\in\N}\subseteq\ell^2(\N)$ is a Cauchy-sequence
    w.r.t.\ $\ell_s^2(\N)$, and since $\ell^2_s(\N)$ is closed,
    there exists $\bsx\in\ell^2_s(\N)\subseteq\ell^2(\N)$ such that
    $\cE(x_n)\to \bsx\in\ell^2_s(\N)$. 
    Since
    $\cE(\cX)\subseteq\ell^2(\N)$ is closed, 
    $\bsx$ belongs to $\cE(\cX)\subseteq\ell^2(\N)$.
    Since $\cD$ maps from $\ell^2(\N)$ to $\cX$,
    $\tilde x:=\cD(\bsx)\in\cX$ is well-defined and belongs to $\cX^s$
    since
    \begin{equation*}
      \norm[\cX^s]{\tilde x} = \norm[\ell^2_s]{\bsx} < \infty.
    \end{equation*}
    Using that
      $\cD\circ\cE$ is the identity on $\cX$
      (cp.~\eqref{eq:compid}), we find that
      $\cE\circ\cD\circ\cE =\cE$ and thus 
      $\cE\circ\cD$ is the identity on $\cE(\cX)$. 
    Hence
    \begin{equation*}
      \norm[\cX^s]{x_n-\tilde x}
      =\norm[\ell^2_s]{\cE(x_n)-\cE(\tilde x)}
      =\norm[\ell^2_s]{\cE(x_n)-\cE(\cD(\bsx))}
      =\norm[\ell^2_s]{\cE(x_n)-\bsx}
      \to 0
    \end{equation*}
    as $n\to\infty$. This shows that $\cX^s$ is closed w.r.t.\
    $\norm[\cX^s]{\cdot}$.
  \end{proof}

\begin{remark}\label{rmk:weighted}
  For ONBs $\bsPsi_\cX = \{\psi_j\}_{j\in\N}$ and
  $\bsPsi_\cY = \{ \eta_j \}_{j\in \N}$ of $\cX$ and $\cY$, the
  sequences $(w_j^s\psi_j)_{j\in\N}$ and $(w_j^t\eta_j)_{j\in\N}$ form
  ONBs of $\cX^s$, $\cY^t$ respectively.
\end{remark}

The Hilbert spaces $\cX^s$ and $\cY^t$ are included in their (unique,
\cite{MunozCplxB}) complexified versions
$\cX_\C^s = \{1,i\}\otimes \cX^s$ and
$\cY_\C^t = \{1,i\}\otimes \cY^t$, for which the encoder and decoder
in \eqref{eq:ED} act on weighted, complex sequence spaces.

\begin{remark} \label{rmk:CplxCoeff} 
  The results of this paper can be
  extended to the case where $\cX$ and $\cY$ are separable Hilbert
  spaces over the coefficient field $\C$. 
  We do not elaborate details
  in order to avoid having to distinguish between two cases in the following.
\end{remark}
\begin{remark}\label{remk:SplWavD}
{
The biorthogonal spline wavelet bases constructed in \cite{DahmStevEBE} 
are continuous, piecewise polynomial,
and are stable in the Sobolev spaces $H^s(\domain)$ for $|s|\leq 3/2$. 
Scaling $\psi_j$ obtained in \cite{DahmStevEBE} so that the Riesz basis
property holds in $\cX = L^2(\domain)$, 
appropriate choices of the weight sequence $w_j$ provides 
in particular \eqref{eq:XsYt} with $0\leq s \leq 3/2$ in $\cX^s = H^s(\domain)$.
In addition, the wavelet constructions $\psi_j$ and $\tilde{\psi}_j$ 
in \cite{DahmStevEBE}
can furnish wavelet systems with vanishing moments of any a-priori required 
polynomial order.
}
\end{remark}
\section{Main results}
\label{sec:MainReS}
Our goal is to approximate maps $\cG$ from (subsets of)
$\cX$ into $\cY$. In the following denote
by $\bsPsi_\cX$, $\bsPsi_\cY$ fixed frames of the separable Hilbert
spaces $\cX$, $\cY$ as in \eqref{eq:psieta}, and let the encoder
$\cE:\cX\to\ell^2(\N)$, decoder $\cD:\ell^2(\N)\to\cY$ be as in \eqref{eq:ED}.
With
    $$
    U:=[-1,1]^\N,
    $$
    and 
    $\bsw$ as in Section \ref{sec:Scales}, for $s > \frac12$,
    {and a scaling factor $r>0$} set
\begin{equation}\label{eq:sigma}
  \sigma_r^s:=\begin{cases}
    U\to %
    \cX\\
    \bsy\mapsto r \sum_{j\in\N}w_j^s y_j\psi_j.
    \end{cases}
  \end{equation}
  The condition $s> \frac12$ ensures that the coefficient
    sequence $(rw_j^sy_j)_{j\in\N}$ belongs to $\ell^2(\N)$ so that
    $\sigma_r^s$ is well-defined as a mapping from $U$ to $\cX$
    (cp.~\eqref{eq:calFp}). %
    For $s> \frac{1}{2}$
    we then introduce the following ``Cubes'' in $\cX$ 
\begin{equation*}%
  \tilde C_r^s(\cX):=\setlr{\sigma_r^s(\bsy)}{\bsy\in U}
\end{equation*}
and additionally for $s\geq 0$
\begin{align}\label{eq:Crs}
  C_r^s(\cX)&:=\set{a\in\cX}{\cE(a)\in\times_{j\in\N}[-rw_j^{s},rw_j^{s}]}\nonumber\\
              &=
  \setlr{a\in\cX}{\sup_{j\in\N}|\dup{a}{\tilde{\psi}_j}|w_j^{-s}\le r}.
\end{align}
The sets $C_r^s(\cX)$ will serve as the
domains on which $\cG$ is to be approximated.

\begin{remark}\label{rmk:riesz}
  Observe
  $C_r^s(\cX)\subseteq\tilde C_r^s(\cX)$.
  If $\bsPsi_\cX=(\psi_j)_{j\in\N}$ is a Riesz basis, then (due to the
  basis property) the $\ell^2(\N)$-sequence of expansion coefficients
  of any element $a\in\cX$ w.r.t.\ $\bsPsi_\cX$ is unique. Due to
  \eqref{eq:compid} and \eqref{eq:sigma} it thus must hold
  $\cE(\sigma_r^s(\bsy))=(rw_j^sy_j)_{j\in\N}$ for all $\bsy\in U$.
  This implies $\tilde C_r^s(\cX)= C_r^s(\cX)$.
  For frames, this
  does not hold in general, since $\cE(\sigma_r^s(\bsy))$ 
  need not belong to
  $\times_{j\in\N}[-rw_j^s,rw_j^s]$ for $\bsy\in U$.
\end{remark}
\begin{remark}
\label{rmk:Cs_in_Xs_prime}
Let $s'\ge 0$ and $s>s'+1/2$.
Then
$C^s_r(\cX) \subset \cX^{s'}$, since
for $a\in C^s_r(\cX)$
\begin{equation*}
\| a \|^2_{\cX^{s'}}
=
\sum_{j\in\bbN}
\langle a, \tilde\psi_j \rangle ^2 w_j^{-2s'}
\leq 
r^2 \sum_{j\in\bbN}
 w_j^{2(s-s')}
 <\infty, 
\end{equation*}
due to $(w_j)_{j\in\N}\in\ell^{1+\eps}(\N)$ for any $\eps>0$.
\end{remark}
We shall work under the \emph{assumption that $\cG$ allows a complex
differentiable extension to some open superset of 
$\tilde C_r^{s}(\cX)$ in $\cX_\C$}:

\begin{assumption}\label{ass:1}
  There exist $s>1$, $t>0$, $M<\infty$ and an open set $O_\C\subseteq \cX_\C$
  containing $\tilde C_r^{s}(\cX)$ such that
  $\sup_{a\in O_\C}\norm[\cY_\C^t]{\cG(a)}\le M$
  and $\cG:O_\C \to \cY_\C$ is holomorphic.
\end{assumption}
{
\begin{remark}
  Assumption \ref{ass:1} is for instance satisfied by solution
  operators corresponding to second order elliptic PDEs. Specific
  examples will be discussed in full detail in Section \ref{sec:Expl}
  ahead.
\end{remark}
}

  We emphasize that in Assumption \ref{ass:1} the holomorphy condition
  is merely required w.r.t.\ the topology of $\cY_\C$ and not with
  respect to the stronger topology of $\cY_\C^t$. {We will
  see in Lemma \ref{lemma:Gt} below, that the assumed boundedness}
  already implies holomorphy of $\cG:O_\C\to\cY_\C^{t'}$ for any $t'\in [0,t)$.

{
\subsection{Universality}
\label{sec:Univ}
We observe that, in the present setting \eqref{eq:GenForm}, there holds
a version of the universal approximation theorem. 
We consider admissible activation functions as in \cite{LESHNO1993861}:
         \begin{align*}
          \cA:=\big\{
                        &\text{$\sigma\in L_{\rm loc}^\infty(\R)$ is not polynomial and the closure } \\
                        &\qquad\text{of the points of discontinuity has Lebesgue measure $0$}\big\}.
        \end{align*}
Then there holds:
        \begin{theorem}\label{thm:Universality}
          Let $\cX$, $\cY$ be two separable Hilbert spaces, let
          $\cG:\cX\to\cY$ be continuous and let $\sigma\in\cA$.

          Then there exists a sequence of operator nets
          $\widetilde \cG_n:\cX\to\cY$, $n\in\N$, with architecture
          \eqref{eq:GenForm} such that
          \begin{equation*}
          \forall x \in \cX: \quad 
            \lim_{N\to\infty}\widetilde\cG_N(x)=\cG(x) \;.
          \end{equation*}
          The convergence is uniform on every compact subset of $\cX$.
        \end{theorem}
        The proof is given in \cite[Appendix]{LipONet}.}
\subsection{Worst-case error for NN operator surrogates}
\label{sec:WCENNOp}
Our first main result states that a
holomorphic operator $\cG$ as in Assumption \ref{ass:1} can be
uniformly approximated on $C_r^s(\cX)$ by a NN surrogate of the form
$\cD\circ\tilde\bG\circ\cE$, where $\tilde\bG$ is a ReLU NN.  
More precisely, $\tilde\bG$ is a {map} of the form
\begin{equation}\label{eq:NN}
  {\tilde\bG =} A_L\circ{\rm ReLU}\circ A_{L-1}\cdots
  \circ A_{1}\circ {\rm ReLU}\circ A_0
\end{equation}
where the application of ${\rm ReLU}(x):=\max\{0,x\}$ is understood
componentwise, and each $A_j:\R^{n_j}\to\R^{n_{j+1}}$ is an affine
transformation of the form $A_j(x)=W_jx+b_j$ with
$W_j\in\R^{n_{j+1}\times n_j}$, $b_j\in\R^{n_{j+1}}$.  
The entries of
the weights $W_j$ and the biases $b_j$ are parameters that
determine the NN.  It is common practice to determine these
parameters by NN ``training'', where some regression procedure on
input-output data pairs with the map induced by $\tilde\bG$ of the
form \eqref{eq:NN} is used to find choices of the weights and
biases. 
Alternatively, concrete constructions of 
the NN weights $W_j$ and biases $b_j$ 
based on \emph{a-priori specified samples of input-output data pairs} 
are sometimes proposed (e.g. \cite{HOS22_2887}).  
The presently developed, constructive proofs are of this type.  
We refer to
the number of nonzero entries of all $W_j$ and $b_j$, i.e.
\begin{equation}\label{eq:size}
  {\rm size}(\tilde\bG) := \sum_{0\leq j \leq L } \| W_j \|_0 + \| b_j \|_0
\end{equation}
as the \emph{size of the NN in \eqref{eq:NN}}. In other words,
the ``size'' of the network is the number of trainable network
parameters.

\begin{remark}\label{rmk:padding}
  Any realization of a NN $\tilde\bG$ of the form \eqref{eq:NN}
  represents a map from $\R^{n_0}\to\R^{n_{L+1}}$.  Throughout we will
  also understand $\tilde\bG$ as map from
  $\ell^2(\N)\to\ell^2(\N)$: 
  Let
  {$\tilde A_0\in \R^{n_1\times\infty}$ %
  and
$\tilde A_L \in \R^{\infty\times n_L}$
    be obtained from $A_0$ and $A_L$ by padding the (infinitely many)
    entries with zeros according to
    \begin{equation*}
      \tilde A_0 = \begin{pmatrix}
        (A_0)_{11}&\dots &(A_0)_{1 n_0}&0&\dots\\
        \vdots &\ddots &\vdots&\vdots&\dots\\
        (A_0)_{n_1 1}&\dots &(A_0)_{n_1 n_0}&0&\dots\\
      \end{pmatrix},
      \qquad
      \tilde A_L=
\begin{pmatrix}
        (A_L)_{11}&\dots &(A_L)_{1 n_{L}}\\
        \vdots &\ddots &\vdots\\
        (A_L)_{n_{L+1} 1}&\dots &(A_L)_{n_{L+1} n_{L}}\\
        0&\dots &0\\
        \vdots&\vdots&\vdots
      \end{pmatrix}.
    \end{equation*}
    Formally replacing $A_0$, $A_L$ by $\tilde A_0$, $\tilde A_L$
    in \eqref{eq:NN}, $\tilde\bG$ in \eqref{eq:NN}} becomes a mapping from
  $\ell^2(\N)\to\ell^2(\N)$. %
This new network is of the same size as the original one, since we
only add zero entries (cp.~\eqref{eq:size}). 
The function it realizes is
obtained by padding the original network input and output with zeros.
\end{remark}
The composition
$\cD\circ\tilde\bG\circ\cE$ is well-defined in the sense
of Rmk.~\ref{rmk:padding}.

\begin{theorem}
  \label{thm:main}
  Let Assumption \ref{ass:1} be satisfied with $s>1$, $t>0$.  
  Fix $\delta>0$ (arbitrarily small) and $r>0$.
  Then there exists a constant $C > 0$ %
  such that
  for every $N\in\N$ there exists a ReLU NN
  $\tilde{\bG}_N$ of size $O(N)$ such that
  \begin{equation}\label{eq:mainlinf}
    \sup_{a\in C_r^s(\cX)}\norm[\cY]{\cG(a)- \cD({\tilde{\bG}}_N(\cE(a)))}
    \le 
    C N^{-\min\{s-1,t\}+\delta}.
  \end{equation}
\end{theorem}

Next, introduce the closed ball of radius $r$ in $\cX^s$
\begin{equation*}
  B_r(\cX^s):=\setlr{a\in\cX}{\norm[\cX^s]{a}\le r}.
\end{equation*}
Since for any $\eps>0$, we have
\begin{equation*}
  B_r(\cX^s)\subseteq C_r^s(\cX)\subseteq B_{r_\eps}(\cX^{s-\frac{1}{2}-\eps})
\end{equation*}
with $r_\eps:=r(\sum_{j\in\N}w_j^{1+2\eps})^{1/2}<\infty$
(cp.~\eqref{eq:XsYt}, \eqref{eq:Crs} and
Rmk.~\ref{rmk:Cs_in_Xs_prime}), we trivially get the following:
\begin{corollary}\label{cor:ball}
  Consider the setting of Theorem \ref{thm:main}. 
  Then there exists a constant $C>0$ such that for every $N\in \bbN$
  there exists a ReLU NN %
  $\tilde{\bG}_N$ of size $O(N)$ such that
  \begin{equation*}
    \sup_{a\in B_r(\cX^s)}\norm[\cY]{\cG(a)-\cD(\tilde{\bG}_N(\cE(a)))} 
    \le C N^{-\min\{s-1,t\}+{\delta}}.
  \end{equation*}  
\end{corollary}

\begin{remark}
Clearly $B_r(\cX^s)$ is a proper subset of $C_r^s(\cX)$.  However, in
general there is no $s'>s$ and $r'>0$ such that
$B_{r'}(\cX^{s'})\subseteq C_r^s(\cX)$, thus we cannot use
Theorem \ref{thm:main} to improve the convergence rate on the ball 
$B_r(\cX^s)\subset \cX^s$.
\end{remark}
Our results provide sufficient conditions under which operator nets
can overcome the curse of dimensionality, since we allow the
operators to have infinite dimensional domains. 
The proof hinges on
certain ``sparsity'' properties of the encoded coefficients. Neural
networks are able to exploit this form of intrinsic low-dimensionality
and in this way elude the curse of dimension.  However, we emphasize
that it is the intrinsic sparsity of the considered functions, rather
than specific properties of NNs that lead to these
statements.  The same convergence rate %
can
be obtained with other
methods such as sparse-grid polynomial interpolation or low-rank
tensor approximation, as we discuss in Section \ref{sec:wcgpc} ahead.

\subsection{Mean-square error for NN operator surrogates}
\label{sec:mnsq}
We can improve the operator approximation rate of Theorem
\ref{thm:main}, if we measure the error in a mean-square sense. To
this end assume that $\bsPsi_\cX$ is a Riesz basis, and let
$\mu:=\otimes_{j\in\N}\frac{\lambda}{2}$ be the uniform probability
measure on $U:=\times_{j\in\N}[-1,1]$ equipped with its %
product Borel sigma algebra, where $\lambda$ stands for the Lebesgue
measure in $\bbR$. 
By Rmk.~\ref{rmk:riesz},
the pushforward $(\sigma_r^s)_\sharp \mu$ of $\mu$ under $\sigma_r^s$
then constitutes a measure on $C_r^s(\cX)$.

\begin{theorem}%
  \label{thm:main_riesz}
  Assume that $\bsPsi_\cX$ is a Riesz basis.  Let Assumption
  \ref{ass:1} be satisfied with $s>1$, $t>0$.  Fix $\delta>0$
  (arbitrarily small) and $r>0$.
  Then there exists a constant $C > 0$ %
  such that
  for every $N\in\N$ there exists a ReLU NN
  $\tilde{\bG}_N$ of size $O(N)$ such that
  \begin{equation}\label{eq:mainl2}
    \normlr[L^2(C_r^s(\cX),(\sigma_r^s)_\sharp\mu;\cY)]{\cG - \cD\circ {\tilde{\bG}}_N \circ \cE}
    \le 
    C N^{-\min\{s-\frac{1}{2},t\}+\delta}.
  \end{equation}
\end{theorem}

\subsection{Worst-case error for {spectral} operator surrogates}
\label{sec:wcgpc}
For our third main result, instead of a NN $\tilde\bG_N$ we use a
multivariate polynomial $p_N:\R^{n}\to\R^{m}$. The operator surrogate
then takes the form $\cD\circ p_N\circ \cE$, where the composition is
again understood as truncating the output of $\cE$ after the first $n$
parameters, and padding the output of $p_N$ with infinitely many zeros
(cp.~Rmk.~\ref{rmk:padding}). 
The advantage over the NN operator
surrogate in Theorem \ref{thm:main} is that, while we achieve the
same converence rate, the proof is constructive, 
and one can explicitly compute $p_N$ as an interpolation polynomial,
{based on a finite set of \emph{judiciously chosen} 
(rather than random, i.i.d) input-output pairs.}
Hence, the ``training'' %
{consists of an explicit and deterministic construction,
rather than the minimization of a (typically non-convex)
loss by stochastic optimization methods.
Moreover, and as a consequence,
we are able to obtain} (higher-order) deterministic {(``worst-case'')} 
generalization bounds, 
{
rather than (low-order) probabilistic bounds 
as commonly applied within statistical learning theory.
}

Contrary to Section \ref{sec:mnsq}, 
we now allow for $\bsPsi_\cX$ again to be a frame.

\begin{theorem}
  \label{thm:main2} 
  Consider the setting of Theorem \ref{thm:main}. Then
  there is a constant $C>0$ such that for every $N\in\N$ there exists
  a multivariate polynomial $p_N$
  such that %
  \begin{equation*}
    \sup_{a\in C_r^s(\cX)}\norm[\cY]{\cG(a)-\cD(p_N(\cE(a)))}
    \le 
     C N^{-\min\{s-1,t\}+\delta}.
  \end{equation*}

  Furthermore, $p_N$ belongs to an $N$-dimensional space of
  multivariate polynomials. Its components are interpolation
  polynomials, whose computation requires the evaluation of
  $\dup{\cG(a)}{\tilde{\eta}_j}$ at at most $N$ tuples
  $(a,j)\in C_r^s(\cX)\times\N$.
\end{theorem}
{Prior to proving these results}, 
let us make one further remark.
Corollary \ref{cor:ball} states that we can \emph{uniformly}
approximate any holomorphic $\cG$ as in Assumption \ref{ass:1} with a
NN operator surrogate on the ball $B_r(\cX^s)$; an analogous corollary
also holds for {spectral} operator surrogates. 
Since $\cX^s$ is an infinite dimensional Hilbert space, 
$B_r(\cX^s)$ is not compact in $\cX^s$. 
Thus the image of $B_r(\cX^s)$ will in general also not be
compact in $\cY^t$, for example in case $\cX=\cY^t$ and
$\cG:\cX\to\cY^t$ is the identity (which satisfies Assumption
\ref{ass:1}). Therefore it seems counterintuitive that we can
uniformly approximate $\cG$ using a NN with only \emph{finitely} many
parameters (or a polynomial of finite degree). This is possible,
because the approximation rate is stated not in the norm of $\cY^t$,
but in the weaker norm of $\cY$, and in fact a ball in $\cY^t$ is
compact in $\cY$:
\begin{lemma}
  For every $0\le t'<t<\infty$, the set $B_r(\cY^t)$ is compact in $\cY^{t'}$.
\end{lemma}
\begin{proof}
  Let $(a_n)_{n\in\N}$ be a sequence in $B_r(\cY^t)$ and denote
    $\bsx_n=\cE(a_n)$, where
    $\bsx_n=(x_{n,j})_{j\in\N}\in\ell^2(\N)$. 
    Then, due to
    $a_n\in B_r(\cY^t)$ holds $\sum_{j\in\N}x_{n,j}^2w_j^{-2t}\le r$
    for all $n\in\N$ and, in particular,
    \begin{equation*}
      x_{n,j} \in [-rw_j^t,rw_j^t]\qquad\forall n,j\in\N.
    \end{equation*}

    Compactness of $[-rw_1^t,rw_1^t]$ implies the existence
    of a subsequence $(\bsx_{1,n})_{n\in\N}$ of
    $(\bsx_n)_{n\in\N}$, such that $(x_{1,n,1})_{n\in\N}$ is a Cauchy
    sequence in $[-rw_1^t,rw_1^t]$. Inductively, let
    $(\bsx_{k,n})_{n\in\N}$ be a subsequence of
    $(\bsx_{k-1,n})_{n\in\N}$ such that $(x_{k,n,k})_{n\in\N}$ is a
    Cauchy sequence in $[-rw_k^t,rw_k^t]$. Then
    $\tilde \bsx_n:=\bsx_{n,n}$ defines a subsequence of
    $(\bsx_n)_{n\in\N}$ with the property that
    $(\tilde x_{n,j})_{n\in\N}$ is a Cauchy sequence for each $j\in\N$.  
    Denote the corresponding sequence in $\cY^t$ by $(\tilde a_n)_{n\in\N}$.

    Now fix $\eps>0$ arbitrary. 
    Let $N(\eps)\in\N$ be so large that
    $w_{N(\eps)}^{2(t-t')}<\frac{\eps}{4r^2}$. 
    Then for all $n\in\N$
    \begin{equation*}
      \sum_{j>N(\eps)}x_{n,j}^2w_{j}^{-2t'}
      \le w_{N(\eps)}^{2(t-t')}
      \sum_{j>N(\eps)}x_{n,j}^2w_{j}^{-2t}\le \frac{\eps}{4}.
    \end{equation*}
    Next,
    since $(\tilde x_{n,j})_{n\in\N}$ is a Cauchy sequence for each
    $j\le N(\eps)$, there exists $M(\eps)\in\N$ so large that
    \begin{equation*}
      \sum_{j=1}^{N(\eps)}|x_{m,j}-x_{n,j}|^2w_j^{-2t'}<\frac{\eps}{2}\qquad
      \forall m,n\ge M(\eps).
    \end{equation*}
    Then for all $m$, $n\ge M(\eps)$ we have
  \begin{equation*}
    \norm[\cY^{t'}]{\tilde a_n-\tilde a_m}\le
    \sum_{j=1}^{N(\eps)}|x_{m,j}-x_{n,j}|^2w_j^{-2t'}
    +2\sum_{j>N(\eps)}(x_{m,j}^2+x_{n,j}^2)w_j^{-2t'}
    \le\eps.
  \end{equation*}
  Thus $(\tilde a_n)_{n\in\N}$ is a Cauchy sequence in
  $\cY^{t'}$. This concludes the proof.
\end{proof}

\section{Proof of Theorem \ref{thm:main} and Theorem
  \ref{thm:main_riesz}}\label{sec:proof}
With $s>1$ and $\sigma_r^s$ as in \eqref{eq:sigma} 
let in the following
  \begin{equation}
    \bG:=\tilde\rF_\cY\circ\cG\circ\sigma_r^s.
  \end{equation}
  Since $\sigma_r^s:U\to\cX$, $\cG:\cX\to\cY$ and
  $\tilde\rF_\cY:\cY\to\ell^2(\N)$ we have $\bG:U\to\ell^2(\N)$.
  In addition, due to $\cD={\rF}_\cY'$ and \eqref{eq:compid} we have that
  $\cD\circ \tilde\rF_\cY$ is the identity on $\cY$ and thus
\begin{equation}\label{eq:u}
  u(\bsy) := ({\calD} \circ \bG)(\bsy) = \cG(\sigma_r^s(\bsy))
\end{equation}
is well-defined for $\bsy\in U$.
To prove
Theorem \ref{thm:main} we  %
first show that $\bG:U\to{\ell^2}$ can be approximated.
\subsection{Auxiliary results}
{
  We start by showing that the functions $\cG:O\to\cY_\C$
  in Assumption \ref{ass:1} are holomorphic w.r.t.\
  the topologies of $\cY_\C^{t'}$ for any $t'<t$.
\begin{lemma}\label{lemma:Gt}
    Let Assumption \ref{ass:1} hold. Then for every
  $t'\in [0,t)$, the map $\cG:O_\C\to \cY^{t'}_\C$ is holomorphic.
\end{lemma}

To prove Lemma \ref{lemma:Gt}, we use the following result:
\begin{proposition}\label{prop:conthol}
  Let $X$, $Y$, $Z$ be three complex Banach spaces and
  $O\subseteq X$ open, nonempty.
  Let $\iota:Z\to Y$ be a %
  bounded injective linear map. 
  Then the following statements are equivalent:
  \begin{enumerate}
  \item\label{item:Gt1} 
    $\iota\circ\cG:O\to Y$ is holomorphic
    and 
    $\cG:O \to Z$ is continuous,
  \item\label{item:Gt2} 
   $\cG:O\to Z$ is holomorphic.
  \end{enumerate}
\end{proposition}
\begin{proof}
  \ref{item:Gt2} $\Rightarrow$ \ref{item:Gt1}: 
  Holomorphy of $\cG:O\to Z$ implies its continuity, 
  and the composition of
  the holomorphic map $\cG:O\to Z$ with the bounded linear
  (thus holomorphic) map $\iota:Z\to Y$ is holomorphic 
  (as a composition of holomorphic functions).

  \ref{item:Gt1} $\Rightarrow$ \ref{item:Gt2}: Fix a point $x\in O$
  and let $r>0$ be so small that the ball of radius $2r$ around $x$ in
  $X$ is contained in $O$.  
  Additionally fix $h\in X$ with
  ${0<}\norm[X]{h}\le 1$
  and $\xi\in\C$ with $|\xi|<r$ 
  and consider the Bochner integral
  \begin{equation}\label{eq:CauchyBochner}
    \frac{1}{2\pi\ii}\int_{\set{\zeta\in\C}{|\zeta|=r}} \frac{\cG(x+\zeta h)}{\zeta-\xi}\dd\zeta\in Y',
  \end{equation}
  which is understood w.r.t.\ the arclength measure on the
  circle $C_r:=\set{\zeta\in\C}{|\zeta|=r}$, oriented counterclockwise.

Continuity of $\cG: O\to Z$ implies that
$\zeta\mapsto \cG(x+\zeta h)\in Z$ is measurable 
w.r.t.\ Borel $\sigma$-algebra on $C_r$, and that
$\set{\cG(x+\zeta h)}{\zeta\in C_r}\subseteq Z$ is 
a compact (and thus separable) subset of $Z$. 
Hence \eqref{eq:CauchyBochner} is well-defined as a Bochner integral. 
Using (i) the fact that bounded
linear mappings may be exchanged with the Bochner integral and 
(ii) Cauchy's integral formula (for the holomorphic map
$\iota\circ\cG:O\to Y$) in complex Banach spaces, see
\cite[13.3]{SBchae} or \cite[Thm.~1.2.11]{JZDiss} 
for this specific version, 
it holds
\begin{align*}
    \iota\left(\frac{1}{2\pi\ii}\int_{\set{\zeta\in\C}{|\zeta|=r}} \frac{\cG(x+\zeta h)}{\zeta-\xi}\dd\zeta\right)
    &=\frac{1}{2\pi\ii}\int_{\set{\zeta\in\C}{|\zeta|=r}} \frac{\iota\circ \cG(x+\zeta h)}{\zeta-\xi}\dd\zeta\nonumber\\
    &= \iota\circ \cG(x+\xi h)
\end{align*}
for all complex $|\xi|<r$.

Injectivity of $\iota$ yields
\begin{equation*}
\frac{1}{2\pi\ii}\int_{\set{\zeta\in\C}{|\zeta|=r}} \frac{\cG(x+\zeta h)}{\zeta-\xi}\dd\zeta 
=  
\cG(x+\xi h)\in Z.
\end{equation*}
Continuity of $\cG$ implies its local boundedness, and thus we can
conclude that this function is complex differentiable for $\xi\in\C$
with $|\xi|<r$.  This verifies so-called G\^ateaux holomorphy of
$\cG:O\to Z$.  Together with the continuity of this map, this implies
holomorphy in the sense of complex Fr\'echet differentiability
\cite[14.9]{SBchae}.
\end{proof}

}

\begin{proof}[of Lemma \ref{lemma:Gt}]
  Applying Prop.~\ref{prop:conthol} with 
  $X:=\cX_\C$, $Y:=\cY_\C$ and $Z := \cY_\C^{t'}$, 
  in order to show the lemma it suffices to
  check that $\cG:O_\C\to \cY^{t'}$ is continuous, since we trivially
  have a continuous embedding $\iota:\cY_\C^{t'}\to \cY_\C$.

  To do so fix $\eps>0$ and $x\in O_\C$.  
  We need to show that there
  exists $\delta>0$ such that $\norm[\cX_\C]{x-\tilde x}<\delta$
  implies $\norm[\cY_\C^{t'}]{\cG(x)-\cG(\tilde x)}<\eps$. 
  For any
  $n\in\N$
  \begin{align}\label{eq:cGcont}
    &\norm[\cY_\C^{t'}]{\cG(x)-\cG(\tilde x)}^2
    =\sum_{j\in\N}|\dup{\cG(x)-\cG(\tilde x)}{\tilde\eta_j}|^2w_j^{-2t'}\nonumber\\
    &\qquad\le\sum_{j=1}^n|\dup{\cG(x)-\cG(\tilde x)}{\tilde\eta_j}|^2w_j^{-2t'}
      +2(\sup_{j>n}w_j^{2(t-t')})\sum_{j>n}(|\dup{\cG(x)}{\tilde\eta_j}|^2+|\dup{\cG(x')}{\tilde\eta_j}|^2)w_j^{-2t}
      \nonumber\\
    &\qquad\le\norm[\cY_\C]{\cG(x)-\cG(x')}^2\sum_{j=1}^nw_j^{-2t'}
      +4M^2 (\sup_{j>n}w_j^{2(t-t')}),
  \end{align}
  where we have used that
  $\sum_{j\in\N}|\dup{\cG(a)}{\tilde\eta_j}|^2w_j^{-2t}=\norm[\cY_\C^{t}]{\cG(a)}^2\le
  M^2$ for all $\tilde x\in O$ by Assumption \ref{ass:1}.  Since
  $(w_j)_{j\in\N}$ is a null-sequence, by choosing $n$ large enough,
  the second term on the right-hand side of \eqref{eq:cGcont} is less
  than $\eps/2$. Since $\cG:O_\C\to\cY_\C$ is continuous, there exists
  $\delta>0$ depending on $\eps$ and $n$ such that
  $\norm[\cX_\C]{x-\tilde x}<\delta$ implies the first term on the
  right-hand side of \eqref{eq:cGcont} to be less than $\eps/2$. 
  This concludes the proof.
\end{proof}

It is well known, that algebraic decay of the ``input'' sequence
$w_j^s\psi_j$ in \eqref{eq:sigma} is inherited by the Legendre
coefficients of the ``output'' $u(\bsy)$, see for example \cite{CCS15}
and the earlier works \cite{CDS2010,CDS2011} for the analysis for some
specific choices of $\cG$, and \cite{cohen_devore_2015} for the
general analysis.  To provide %
a statement of this type, we 
denote by
$L_n$ the $n$-th Legendre polynomial normalized such that
$\frac{1}{2}\int_{-1}^1L_n(x)^2\dd x = 1$. 
In the following $\cF$
stands for the set of inifite dimensional
multiindices with
''finite support''
\begin{equation*}
  \cF:=\set{(\nu_j)_{j\in\N_0}\in\N_0^\N}{|\bsnu|<\infty},
\end{equation*}
{where $|\bsnu|:=\sum_{j\in\N}\nu_j$. 
For such multiindices and
$\bsy=(y_j)_{j\in\N}\in U=[-1,1]^\N$ we write}
\begin{equation*}
  L_\bsnu(\bsy):=\prod_{j\in\supp\bsnu}L_{\nu_j}(y_j),
\end{equation*}
{where $\supp\bsnu:=\set{j\in\N}{\nu_j\neq 0}$ and}
where an empty product is understood as constant $1$.
With the
infinite product measure $\mu=\otimes_{j\in\N}\frac{\lambda}{2}$ 
on $U=[-1,1]^\N$, we then have $\norm[L^2(U,\mu)]{L_\bsnu}=1$.
As is well known, e.g., \cite[Theorem 2.12]{SchwabGittelson2011},
$(L_\bsnu)_{\bsnu\in\cF}$ is an orthonormal basis of
$L^2(U,\mu)$.
Furthermore, there holds the bound
\begin{equation}\label{eq:Lnubound}
  \norm[L^\infty(U)]{L_\bsnu}\le \prod_{j\in\N}(1+2\nu_j)^{1/2}
\end{equation}
for all $\bsnu\in\cF$, see \cite[\S 18.2(iii) and \S 18.3]{nist} with
our normalization of $L_\bsnu$.

We work with the following theorem, which, apart from giving
an algebraically decaying upper bound on the Legendre coefficients,
additionally provides information on the structure of these upper bounds. 
It is essentially \cite[Theorem 2.2.10]{JZDiss} stated in the current
setting. 
To formulate the result, we first introduce an order-relation on multi-indices.
For $\bsmu=(\mu_j)_{j\in\N}$ and $\bsnu=(\nu_j)_{j\in\N}\in\cF$ we
write $\bsmu\le\bsnu$ iff $\mu_j\le\nu_j$ for all $j\in\N$. 
A set $\Lambda\subseteq\cF$ is \emph{downward closed} iff
$\bsnu\in\Lambda$ implies $\bsmu\in\Lambda$ whenever $\bsmu\le\bsnu$.

\begin{theorem}\label{thm:bpe}
  Let Assumption \ref{ass:1} be satisfied with some $s>1$ and
  $t>0$. Fix $\tau>0$, $p\in (\frac{1}{s},1]$ and $t'\in [0,t)$. 
  For $\bsnu\in \cF$, set
  $\omega_\bsnu:=\prod_{j\in\supp\bsnu}(1+2\nu_j)$ for all
  $\bsnu\in\cF$ (empty products are equal to $1$). 

  Then there exists
  $C>0$ and a sequence $(a_\bsnu)_{\bsnu\in\cF}\in\ell^p(\cF)$ of
  positive numbers such that
  \begin{enumerate}
  \item\label{item:abound} for each $\bsnu\in\cF$
    \begin{equation}\label{eq:anu}
      \omega_\bsnu^\tau\normlr[\cY^{t'}]{\int_U L_\bsnu(\bsy)u(\bsy)\dd\mu(\bsy)}\le C a_\bsnu,
    \end{equation}
  \item\label{item:aprop} there exists an enumeration
    $(\bsnu_i)_{i\in\N}$ of $\cF$ such that $(a_{\bsnu_i})_{i\in\N}$
    is monotonically decreasing, the set
    $\Lambda_N:=\set{a_{\bsnu_i}}{i\le N}\subseteq\cF$ is downward
    closed for each $N\in\N$, and additionally
    {
    \begin{equation}\label{eq:md}
      m(\Lambda_N):=\sup_{\bsnu\in\Lambda_N}|\bsnu|=O(\log(|\Lambda_N|)),\qquad
      d(\Lambda_N):=\sup_{\bsnu\in\Lambda_N}|\supp\bsnu|=o(\log(|\Lambda_N|))
    \end{equation}
    }
    as $N\to\infty$,
  \item\label{item:cprop} 
   {there exist $T>1$ and $t>0$ such that with
    \begin{equation} \label{eq:cprop}
      \bsrho=(\rho_j)_{j\in\N}\qquad\text{where}\qquad\rho_j:=\max\{T,tw_j^{s(p-1)}\}
    \end{equation}
    and $e_\bsnu:=\rho^{\bsnu}$ it holds $(a_\bsnu e_\bsnu)_{\bsnu\in\cF}\in\ell^1(\cF)$
    and $(e_\bsnu^{-1})_{\bsnu\in\cF}\in\ell^{p/(1-p)}$,}
  \item\label{item:conv} 
    the following expansion holds with absolute and uniform convergence:
    \begin{equation*}
      \forall \bsy\in U:\quad 
      u(\bsy) = \sum_{\bsnu\in\cF}L_\bsnu(\bsy) \int_U L_\bsnu(\bsx)u(\bsx)\dd\mu(\bsx)\in\cY^{t'}.
    \end{equation*}
  \end{enumerate}
\end{theorem}

\begin{proof}
  By definition $u(\bsy)=\cG(\sum_{j\in\N}y_jw_j^s\psi_j)$ for all
  $\bsy\in U$. 
  Our choice of the sequence $\bsw$ guarantees
  $\sum_{j\in\N}\norm[\cX]{w_j^s\psi_j}^p \le\Lambda_\bsPsi^{p/2}
  \sum_{j\in\N}w_j^{sp}<\infty$, where we used that
    $\norm[\cX]{\psi_j}\le\Lambda_\bsPsi^{1/2}$ by
    \eqref{eq:rmk:FrBd}.
  By Lemma \ref{lemma:Gt} the map $\cG:O_\C\to \cY^{t'}_\C$
  is holomorphic and uniformly bounded in norm by $M$, where by
  Assumption \ref{ass:1} the set
  $\set{\sum_{j\in\N}y_jw_j^s\psi_j}{\bsy\in U}$ is contained in
  $O_\C$. Thus $u(\bsy):=\cG(\sigma_r^s(\bsy))$ satisfies \cite[Assumption
  1.3.7]{JZDiss}.

  Now \cite[Theorem 2.2.10 (i) and (ii)]{JZDiss} with ``$k=1$'' give
  the existence of $(a_\bsnu)_{\bsnu\in\cF}\in\ell^p(\cF)$ satisfying
  item \ref{item:abound} of the current theorem. 
{Item \ref{item:aprop}
  is a consequence of \cite[Theorem 2.2.10 (iii)]{JZDiss} and
  \cite[Lemma 1.4.15]{JZDiss}. We apply \cite[Theorem 2.2.10
  (iii)]{JZDiss} here with the set $\mathfrak{J}$ occurring in
  \cite[Theorem 2.2.10 (iii)]{JZDiss} chosen as $\mathfrak{J}=\N_0$. Item
  \ref{item:cprop} then holds by \cite[Theorem 2.2.10 (iii)]{JZDiss}.}
  Finally \ref{item:conv} holds by \cite[Corollary 2.2.12]{JZDiss}.
\end{proof}

To approximate the bounded function $u:U\to\cY$ in \eqref{eq:u}, we
first expand it in the %
frame
$(\eta_jL_\bsnu(\bsy))_{j,\bsnu}$ of $L^2(U,\mu;\cY)$:
\begin{subequations}\label{eq:expansion}
\begin{equation}
  u(\bsy) = \cG(\sigma_r^s(\bsy)) = \sum_{j\in\N}\sum_{\bsnu\in\cF} c_{\bsnu,j}\eta_j
  L_\bsnu(\bsy)
\end{equation}
with coefficients
\begin{equation}
  c_{\bsnu,j}:=\int_U L_\bsnu(\bsy)
  \dup{\cG(\sigma_r^s(\bsy))}{\tilde{\eta}_j}\dd\mu(\bsy)
\end{equation}
\end{subequations}
with convergence in $L^2(U,\mu;\cY)$. We have the following weighted
bound on these coefficients:

\begin{proposition}\label{prop:weightsum}
  Consider the setting of Theorem \ref{thm:bpe}. 
  Then for each
  $\bsnu\in\cF$
  \begin{equation*}
    \omega_\bsnu^2 \sum_{j\in\N}w_j^{-2t'}c_{\bsnu,j}^2\le C^2 a_\bsnu^2.
  \end{equation*}
\end{proposition}
\begin{proof}
  It holds
  \begin{equation*}
    \normlr[\cY^{t'}]{\int_UL_\bsnu(\bsy) u(\bsy)\dd\mu(\bsy)}^2=
    \sum_{j\in\N}w_j^{-2t'}\duplr{\int_U L_\bsnu(\bsy) u(\bsy)\dd\mu(\bsy)}{{\tilde \eta_j}}^2
    =\sum_{j\in\N}w_j^{-2t'}c_{\bsnu,j}^2.
  \end{equation*}
  Together with %
  \eqref{eq:anu} this gives the statement.
\end{proof}

Before proving the main statement we need one more lemma.

\begin{lemma}\label{lemma:ml}
  Let $\alpha > 1$, $\beta>0$ and assume given two 
  sequences $(a_i)_{i\geq 1}, (d_j)_{j\geq 1}\subset (0,\infty)^\N$
   with 
  $a_i\lesssim i^{-\alpha}$ and $d_j\lesssim j^{-\beta}$ 
  for all $i$, $j\in\N$. 
  Assume that additionally $(d_j)_{j\in\N}$ is monotonically decreasing.
  Suppose that
  there exists $C<\infty$ such that the sequence
  $(c_{i,j})_{j,i\in\N}$ satisfies 
  \begin{equation*}
\forall i\in \N:\quad \sum_{j\in\N}c_{i,j}^2d_j^{-2}\le C^2 a_i^2.
  \end{equation*}
  Then for every $\delta>0$
  \begin{enumerate}
  \item for every $N\in\N$ exists $(m_i)_{i\in\N}\subseteq\N_0$
    monotonically decreasing such that $\sum_{i\in\N}m_i\le N$ and
  \begin{equation*}
    \sum_{i\in\N}\left(\sum_{j> m_i}c_{i,j}^2\right)^{1/2}\lesssim N^{-\min\{\alpha-1,\beta\}+\delta},
  \end{equation*}
\item for every $n\in\N$ exists $(m_i)_{i\in\N}\subseteq\N_0$
  monotonically decreasing
  such that $\sum_{i\in\N}m_i\le N$ and
  \begin{equation*}
    \left(\sum_{i\in\N}\sum_{j> m_i}c_{i,j}^2\right)^{1/2}\lesssim N^{-\min\{\alpha-\frac{1}{2},\beta\}+\delta}.
  \end{equation*}
  \end{enumerate}
\end{lemma}
\begin{proof}
  Fix $n\in\N$. Set
  $m_i:=\lceil(\frac{n}{i})^{(\alpha-1)/\beta}\rceil$ for $i\le n$ and
  $m_i:=0$ otherwise. For $i\le n$, since $(d_j)_{j\in\N}$ was assumed
  monotonically decreasing,
  \begin{equation*}
    \left(\sum_{j>m_i}c_{i,j}^2\right)^{1/2}
    =
    \left(\sum_{j>m_i}c_{i,j}^2d_{j}^2 d_j^{-2}\right)^{1/2}
    \le C d_{m_i} a_i\lesssim
    m_i^{-\beta}i^{-\alpha}
  \end{equation*}
  and for $i>n$ we have
  $(\sum_{j>m_i}c_{i,j}^2)^{1/2}=(\sum_{j\ge 1}c_{i,j}^2)^{1/2}\lesssim a_i\lesssim
  i^{-\alpha}$. Thus
  \begin{align*}
    \sum_{i\in\N}\left(\sum_{j> m_i}c_{i,j}^2\right)^{1/2}
    \lesssim \sum_{i\le n} m_i^{-\beta}i^{-\alpha}+C \sum_{i>n}i^{-\alpha}
    \lesssim \sum_{i\le n} \left(\frac{n}{i}\right)^{-\alpha+1}i^{-\alpha}+n^{-\alpha+1}
    \lesssim n^{-\alpha+1}\log(n).
  \end{align*}
In addition,
  \begin{equation*}
    \sum_{j\in\N}m_j\lesssim n+n^{\frac{\alpha-1}{\beta}}\int_1^{n+1} x^{-\frac{\alpha-1}{\beta}}\dd x
    \lesssim
    n+n^{\frac{\alpha-1}{\beta}}
    \begin{cases}
      1&\text{if }\frac{\alpha-1}{\beta}>1\\
      \log(n) &\text{if }\frac{\alpha-1}{\beta}=1\\
      n^{1-\frac{\alpha-1}{\beta}}&\text{if }\frac{\alpha-1}{\beta}<1
    \end{cases}
    \lesssim
    \begin{cases}
      n^{\frac{\alpha-1}{\beta}}&\text{if }\frac{\alpha-1}{\beta}>1\\
      n\log(n) &\text{if }\frac{\alpha-1}{\beta}=1\\
      n&\text{if }\frac{\alpha-1}{\beta}<1.
    \end{cases}    
  \end{equation*}
  With $M:=\sum_{j\in\N}m_j$ we get
  \begin{equation*}
    \sum_{i\in\N}\left(\sum_{j> m_i}c_{i,j}^2\right)^{1/2}\lesssim
    \begin{cases}
      M^{-\beta+\delta}&\text{if }\alpha-1\ge \beta\\
      M^{-\alpha+1+\delta}&\text{if }\alpha-1<\beta.
    \end{cases}        
  \end{equation*}
  Choosing $n(N)$ appropriately we can guarantee $M(n)\sim N$.

  For the second item
  fix again $n\in\N$
  and set $m_i:=\lceil(\frac{n}{i})^{(2\alpha-1)/(2\beta)}\rceil$
  for $i\le n$ and $m_i:=0$ otherwise. For $i\le n$
  \begin{equation*}
    \sum_{j> m_i}c_{i,j}^2
    =
    \sum_{j> m_i}c_{i,j}^2d_{j}^2 d_j^{-2}
    \le C d_{m_i}^2 a_i^2\lesssim
    m_i^{-2\beta}i^{-2\alpha},
  \end{equation*}
  and for $i>n$ we have
  $\sum_{j> m_i}c_{i,j}^2\lesssim a_i^2\lesssim i^{-2\alpha}$.
  With the same calculation as in the first case (but with
  $\alpha$, $\beta$ replaced by $2\alpha$, $2\beta$ respectively)
  we get with $M:=\sum_{j\in\N}m_j$
  \begin{equation*}
    \sum_{i\in\N}\left(\sum_{j> m_i}c_{i,j}^2\right)\lesssim
    \begin{cases}
      M^{-2\beta+\delta}&\text{if }2\alpha-1\ge 2\beta\\
      M^{-2\alpha+1+\delta}&\text{if }2\alpha-1< 2\beta.
    \end{cases}        
  \end{equation*}
  This concludes the proof.
\end{proof}

\begin{remark} \label{rmk:OptRate}
The {bounds} %
in the above lemma are optimal as can be checked.
\end{remark}
\subsection{Proof in a particular case}
We prove a particular case of Theorem~\ref{thm:main}, with fixed
parameter range $y_j \in [-1,1]$.
\begin{theorem}\label{thm:tg}
  Let Assumption \ref{ass:1} be satisfied for some $s>1$ and $t>0$.
  Fix $\delta>0$ (arbitrarily small).
  
  Then there exists a constant $C>0$ such that for every $N\in\N$
  exists a ReLU NN $\tilde{\bG}_N$  %
  of size $O(N\log(N)^5)$ such that
  \begin{equation}\label{eq:linferr}
    \sup_{\bsy\in U}\normlr[\cY]{\cG(\sigma_r^s(\bsy))-\cD(\tilde{\bG}_N(\bsy))}
    \le 
    C N^{-\min\{s-1,t\}+\delta}
  \end{equation}
  and
  \begin{equation}\label{eq:l2err}
    \normlr[L^2(U,\cY)]{\cG(\sigma_r^s(\bsy))-\cD(\tilde{\bG}_N(\bsy))}
    \le 
    C N^{-\min\{s-\frac{1}{2},t\}+\delta}.
  \end{equation}
\end{theorem}
\begin{proof}
  Let $(a_\bsnu)_{\bsnu\in\cF}$ and the enumeration
  $(\bsnu_i)_{i\in\N}$ be as in Theorem \ref{thm:bpe} (with ``$\tau$''
  in this theorem being $1/2$), so that $(a_{\bsnu_i})_{i\in\N}$ is
  monotonically decreasing and belongs to $\ell^p(\N)$, where we fix
  $p\in (\frac{1}{s},1]$ such that $\frac{1}{p}\ge s-\delta/2$. Note
  that due to $i a_{\bsnu_i}^p\le \sum_{j\in\N}a_{\bsnu_j}^p<\infty$
  this implies $a_{\bsnu_i}\lesssim i^{-1/p} \le i^{-s+\delta/2}$.

  Fix $N\in\N$ and set $\Lambda_N:=\set{\bsnu_j}{j\le
    N}\subset\cF$, which is a downward closed set by Theorem
  \ref{thm:bpe}.  By \cite[Proposition 2.13]{OSZ21}, for every $0<
  \gamma < 1$ there exists a ReLU NN $(\tilde
  L_\bsnu)_{\bsnu\in\Lambda_{N}}$ such that
    \begin{equation*}
      \sup_{\bsy\in U}\sup_{\bsnu\in\Lambda_{N}}|L_\bsnu(\bsy)-\tilde L_\bsnu(\bsy)|\le \gamma,
    \end{equation*}
    and using \eqref{eq:md} one has the bound
    \begin{equation}\label{eq:sizeLnu}
      {\rm size}((\tilde L_\bsnu)_{\bsnu\in\Lambda_N})
      =O(N\log(N)^4\log(1/\gamma))
    \end{equation}
    on the network size.
    The constant hidden in $O(\cdot)$ is independent of
    $N$ and $\gamma$.  In the following fix for $N\in \bbN$, $N\ge
    2$, the accuracy $\gamma:=N^{-s+\frac{1}{2}} \in
    (0,1)$.  With these choices, the right-hand side of
    \eqref{eq:sizeLnu} is $O(N\log(N)^5)$.

    Next fix $t'\in [0,t)$ such that $t'>t-\delta/2$. By
    Proposition \ref{prop:weightsum}
    with $\omega_{\bsnu}:=\prod_{j\in\supp\bsnu}(1+2\nu_j)\ge 1$
    we have for every $i\in\N$
    \begin{subequations}\label{eq:lemmaass}
    \begin{equation}
      \omega_{\bsnu_i}^{1/2} \left(\sum_{j\in\N}w_j^{-2t'} c_{\bsnu_i,j}^2\right)^{1/2}\lesssim a_{\bsnu_i}
      \lesssim i^{-s+\frac{\delta}{2}}.
    \end{equation}
    Due to
    $(w_j^{t'})_{j\in\N}\in\ell^{1/(t-\delta/2)}(\N)$, by the same
    argument as above (using that $(w_j)_{j\in\N}$ was assumed
    monotonically decreasing) it holds
    \begin{equation}
      w_j^{t'}\lesssim j^{-t+\frac{\delta}{2}}.
    \end{equation}
    \end{subequations}
    We now show \eqref{eq:linferr} and \eqref{eq:l2err} separately.

    \begin{enumerate}
    \item Due to \eqref{eq:lemmaass}
      with $\alpha:=s-\delta/2$
    and $\beta:=t-\delta/2$, by Lemma \ref{lemma:ml} we can find
    $(m_i)_{i\in\N}\subset\N_0$ such that $\sum_{i\in\N}m_i\le N$ and
    \begin{equation}\label{eq:suminf}
      \sum_{i\in\N}\omega_{\bsnu_i}^{1/2}\left(\sum_{j> m_i}c_{\bsnu_i,j}^2\right)^{1/2} \le N^{-\min\{s-1,t\}+\delta}.
    \end{equation}
    Now define for $j\in\N$ (where an empty sum is equal to $0$)
    \begin{equation}\label{eq:tgj}
      \tilde g_j(\bsy):=\sum_{\set{i\in\N}{m_i\ge j}} \tilde L_{\bsnu_i}(\bsy) c_{\bsnu,j}.
    \end{equation}
    Then by \eqref{eq:expansion} with $\tilde \bG_N = (\tilde g_j)_{j\in\N}$
    for all $\bsy\in U$
    \begin{align*}
      &\norm[\cY]{\cG(\sigma_r^s(\bsy))-\cD(\tilde g(\bsy))}
      =\normlr[\cY]{\sum_{i,j\in\N}c_{\bsnu_i,j}L_{\bsnu_i}(\bsy)\eta_j
      - \sum_{i\in\N}\sum_{j\le {{m_i}}}c_{\bsnu_i,j}\tilde L_{\bsnu_i}(\bsy)\eta_j}\nonumber\\
      &\qquad\qquad\le
        \normlr[\cY]{\sum_{i\in\N}L_{\bsnu_i}(\bsy)\sum_{j>{{m_i}}}c_{\bsnu_i,j}\eta_j}
        +\normlr[\cY]{\sum_{i\in\N}(L_{\bsnu_i}(\bsy) - \tilde L_{\bsnu_i}(\bsy))\sum_{j\le {{m_i}}}\eta_j c_{\bsnu_i,j}}\nonumber\\
      &\qquad\qquad 
        \le \Lambda_\cY \sum_{i\in\N}\norm[L^\infty(U)]{L_{\bsnu_i}}
        \left(\sum_{j>{{m_i}}}c_{\bsnu_i,j}^2\right)^{1/2}
        + \Lambda_\cY \gamma \sum_{i\in\N}\left(\sum_{j\le {{m_i}}}c_{\bsnu_i,j}^2 \right)^{1/2},
    \end{align*}
    where $\Lambda_\cY$ denotes the upper frame constant in
    \eqref{eq:FrBd} for the frame $\bsPsi_\cY=(\eta_j)_{j\in\N}$,
    cp.~Rmk.~\ref{rmk:FrBd}.
    By \eqref{eq:Lnubound} and \eqref{eq:suminf} the first term is
    $O(N^{-\min\{s-1,t\}+\delta})$ and the second term is
    $O(\gamma) = O(N^{-\min\{s-1/2,t\}})$ which shows the error bound
    \eqref{eq:linferr}.

    The size of $\tilde\bG_N = (\tilde g_j)_{j\in\N}$ in \eqref{eq:tgj}
    is bounded by
    \begin{align*}
      {\rm size}((\tilde L_\bsnu)_{\bsnu\in\Lambda_N})
      + \sum_{j\in\N} |\set{i\in\N}{m_i\ge j}|
      &= O(N\log(N)^5)+\sum_{i\in\N}\sum_{j\le m_i}1\nonumber\\
      &= O(N\log(N)^5),
    \end{align*}
    since $\sum_{i\in\N}m_i\le N$.

    \item
      Due to \eqref{eq:lemmaass}
      with $\alpha:=s-\delta/2$
    and $\beta:=t-\delta/2$, by Lemma \ref{lemma:ml} we can find
    $(m_i)_{i\in\N}\subset\N_0$ such that $\sum_{i\in\N}m_i\le N$ and
    \begin{equation}\label{eq:suminfl2}
      \sum_{i\in\N}\sum_{j> m_i}c_{\bsnu_i,j}^2 \le N^{-\min\{s-\frac{1}{2},t\}+\delta}.
    \end{equation}
    The rest of the calculation is similar as in the first case.
    Let $\tilde g_j$ be as in \eqref{eq:tgj}.
    Then by \eqref{eq:expansion} and because $(L_\bsnu(\bsy)\eta_j)_{\bsnu,j}$
    is %
    {a frame}
    of $L^2(U,\mu;\cY)$
    \begin{align*}
      &\norm[L^2(U,\cY)]{\cG(\sigma_r^s(\bsy))-\cD((\tilde g_j(\bsy))_{j\in\Lambda_{\N,\epsilon}})}
        \le \normlr[L^2(U,\cY)]{\sum_{i\in\N}\sum_{j>m_i} c_{\bsnu,j} L_{\bsnu_i}(\bsy)\eta_j}
      \\
      &\qquad\qquad\quad 
      +  \normlr[L^2(U,\cY)]{\sum_{i\in\N}\sum_{j\le m_i}
         c_{\bsnu,j}\eta_j |L_\bsnu(\bsy)-\tilde L_\bsnu(\bsy)|}\nonumber\\
      &\qquad\qquad
       \le \Lambda_\cY \left(\sum_{i\in\N}\sum_{j>m_i}c_{\bsnu,j}^2 \right)^{1/2}
         + \Lambda_\cY \gamma \left(\sum_{i\in\N}\sum_{j\le m_i}c_{\bsnu_i,j}^2\right)^{1/2},
    \end{align*}
    where $\Lambda_\cY$ denotes again the upper frame constant in
    \eqref{eq:FrBd} for the frame $\bsPsi_\cY=(\eta_j)_{j\in\N}$.
    By \eqref{eq:suminfl2} the first term is
    $O(N^{-\min\{s-1/2,t\}+\delta})$ and the second term is
    $O(\gamma)=O(N^{-\min\{s-1/2,t\}})$, which shows the error bound
    \eqref{eq:l2err}.

    The size of $\tilde \bG_N$ is bounded in the same way as in the first case.
\end{enumerate}
\end{proof}
\subsection{Proof of Theorem \ref{thm:main} in the general case}\label{sec:gen1}
We obtain Theorem~\ref{thm:main} from Theorem~\ref{thm:tg} by a
scaling argument, via the weight sequences which characterize the
admissible input data.

Introduce the scaling %
\begin{equation}\label{eq:scaling}
S=\times_{j\in\N}[-rw_j^s,rw_j^s]\to U:
(x_j)_{j\in\N}\mapsto \Big(\frac{x_j}{rw_j^s}\Big)_{j\in\N},
\end{equation}
where
$U=[-1,1]^\N$. 
Then (cp.~\eqref{eq:ED} and \eqref{eq:Crs})
    \begin{equation*}
      S\circ\cE(a)\in U\qquad\forall a\in C_r^s(\cX).
    \end{equation*}
  Let $\tilde\bG_N:\ell^2(\N)\to\ell^2(\N)$ be as in Theorem \ref{thm:tg}. 
  Since $\sigma_r^s:U\to \tilde C_r^s(\cX)$ is
  surjective and $\tilde C_r^s(\cX)\supseteq C_r^s(\cX)$ (see
  Rmk.~\ref{rmk:riesz}), \eqref{eq:linferr} in Theorem \ref{thm:tg}
  implies with $\hat \bG_N:= \tilde \bG_N \circ
  S$ %
$$
\begin{array}{l}
\displaystyle
      \sup_{a\in C_r^s(\cX)}\norm[\cY]{\cG(a)-\cD(\hat \bG_N(\cE(a)))}
      \\
\displaystyle
      =\sup_{\set{\bsy\in U}{\sigma_r^s(\bsy)\in C_r^s(\cX)}} 
          \norm[\cY]{\cG(\sigma_r^s(\bsy))-\cD(\tilde \bG_N(S(\cE(\sigma_r^s(\bsy))))} 
        \nonumber\\
      \le\sup_{\bsy\in U}\norm[\cY]{\cG(\sigma_r^s(\bsy))-\cD(\tilde \bG_N(\bsy))} 
      = O(N^{-\min\{s-1,t\}+\delta}).
  \end{array}
$$
    Since $S$ is an infinite linear diagonal transformation, $\hat\bG_N$
    is a network of the same size as $\tilde\bG_N$ and thus of size
    $O(N\log(N)^5)$ by Theorem \ref{thm:tg}
    (cp.~Rmk.~\ref{rmk:padding} and \eqref{eq:NN}).

    Setting $M=M(N):=N\log(N)^5$ we obtain a network of size
    $O(M)$ that achieves error
    $O(N^{-\min\{s-1,t\}+\delta})=O(M^{-\min\{s-1,t\}+2\delta})$.
    Since $\delta>0$ is arbitrary here, we obtain \eqref{eq:mainlinf}. 
    The calculation for \eqref{eq:mainl2} is similar.

\subsection{Proof of Theorem \ref{thm:main_riesz} in the general case}
The argument is similar as in Sec.~\ref{sec:gen1}. Let $S$ be as in
  \eqref{eq:scaling}.  Since $\bsPsi_\cX$ is assumed to be a Riesz
  basis, by Rmk.~\ref{rmk:riesz} it holds (cp.~\eqref{eq:sigma} and
  \eqref{eq:scaling})
  \begin{equation*}
    C_r^s(\cX)=\set{\sigma_r^s(\bsy)}{\bsy\in U}\qquad
    \text{and}\qquad
    \cE(\sigma_r^s(\bsy))=S^{-1}(\bsy)\quad \forall\bsy\in U.
  \end{equation*}

  With $\tilde\bG_N$ as in Thm.~\ref{thm:tg} and $\hat\bG_N:=\tilde\bG_N\circ S$
  we find with \eqref{eq:l2err}
    \begin{align*}
      \norm[L^2(C_r^s(\cX),(\sigma_r^s)_\sharp\mu)]{\cG(a)-\cD(\hat \bG_N(\cE(a)))}
      &=
        \norm[L^2(U,\mu)]{\cG(\sigma_r^s(\bsy))-\cD(\hat \bG_N(\cE(\sigma_r^s(\bsy))))}\nonumber\\
      &=\norm[L^2(U,\mu)]{\cG(\sigma_r^s(\bsy))-\cD(\tilde \bG_N(\bsy))}\nonumber\\
      &=O(N^{-\min\{s-\frac12,t\}+\delta}).
    \end{align*}
  
    The size on the bound of $\hat\bG_N$ follows by the same argument as in Sec.~\ref{sec:gen1}.
  \section{Sparse-grid interpolation}\label{sec:int}
  In this section we discuss operator approximation using sparse-grid
  interpolation instead of NNs. In contrast to neural network
  approximation, the construction of surrogate operators via
  sparse-grid gpc interpolation is, in the current setting, an
  entirely deterministic algorithm of essentially linear complexity,
  which in particular does \emph{not} rely on solving a (nonconvex)
  optimization problem. Further, for the case of uniform
  approximation we will prove the same convergence rate as in Theorem
  \ref{thm:main}.  Thus, from a theoretical viewpoint, sparse-grid
  interpolation has significant advantages over NN training in the
  construction of surrogate operators.
  
  To recall the construction of the Smolyak (sparse-grid) interpolant (e.g. \cite{CCS13_783})
  fix a sequence of distinct points
  $(\chi_j)_{j\in\N_0}\subseteq [-1,1]$. For a multiindex $\bsnu\in\cF$
  and a function $u:U\to\R$ we define for $\bsy=(y_j)_{j\in\N}\in U$
  \begin{equation}\label{eq:Inu}
    (I_\bsnu u)(\bsy) = \sum_{\set{\bsmu\in\cF}{\bsmu\le\bsnu}}u((\chi_{\mu_j})_{j\in\N})
    \prod_{j\in\N}\prod_{\substack{i=0\\ i\neq\mu_j}}^{\nu_j}\frac{y_j-\chi_i}{\chi_{\mu_j}-\chi_i},
  \end{equation}
  where $\bsnu\le\bsmu$ is understood as $\mu_j\le\nu_j$ for all $j\in\N$. 
  The sum in \eqref{eq:Inu} is over
  $\prod_{j\in\N}(1+\nu_j)$ indices, 
  which is finite since $\bsnu\in\cF$.  
  We emphasize that $I_\bsnu$ maps from
  $C^0(U)$ to ${\rm span}\set{\bsy^\bsmu}{\bsmu\le\bsnu}$.  
  Throughout
  we assume that the $(\chi_j)_{j\in\N}$ are such that the Lebesgue
  constant $L((\chi_j)_{j=0}^n)$ of $(\chi_j)_{j=0}^n$ enjoys the
  property
  \begin{equation}\label{eq:lebesgue}
    L((\chi_j)_{j=0}^n)\le (1+n)^\tau\qquad\forall n\in\N_0
  \end{equation}
  for some fixed $\tau>0$.  
  One popular example for such a sequence
  are the so-called Leja points, see \cite{CH13} and the references there.

  For a finite downward closed set $\Lambda\subseteq\cF$
  denote 
  \begin{equation*}
    \bbP_\Lambda:={\rm span}\set{\bsy^\bsnu}{\bsnu\in\Lambda}.
  \end{equation*}
  The Smolyak interpolant is the map
  $I_\Lambda:C^0(U)\to\bbP_{\Lambda}$ defined via
  \begin{equation}\label{eq:ILambda}
    I_\Lambda := \sum_{\bsnu\in\Lambda}\varsigma_{\Lambda,\bsnu} I_\bsnu,\qquad
    \varsigma_{\Lambda,\bsnu}:=\sum_{\set{\bse\in\{0,1\}^\N}{\bsnu+\bse\in\Lambda}}(-1)^{|\bse|}.
  \end{equation}
  \begin{remark}
    It can be checked that the number of function evaluations of $u$
    required to compute $I_\Lambda u$ equals $|\Lambda|$.
  \end{remark}

  The Smolyak interpolant has the following well-known
  properties, see for example \cite[Lemma 1.3.3]{JZDiss}, \cite{CCS13_783}. 
  \begin{lemma}\label{lemma:prop}
    Let $\Lambda$ be finite and downward closed. 
    Then with $\tau$ as in \eqref{eq:lebesgue} and 
    $\omega_\bsnu:=\prod_{j\in\N}(1+2\nu_j)$
    \begin{enumerate}
      \item\label{item:int} $I_\Lambda:C^0(U)\to\bbP_\Lambda$ and $I_{\Lambda}p=p$ for
        all $p\in\bbP_\Lambda$,
      \item\label{item:lebesgue}
        $\norm[L^\infty(U)]{I_{\Lambda}L_\bsnu}\le
        \omega_\bsnu^{3/2+\tau}$ for all $\bsnu\in\cF$.
      \end{enumerate}
  \end{lemma}

  The following theorem shows the same convergence rate as established
  in Theorem \ref{thm:tg} for NNs, for Smolyak interpolation
  $p_N$ of the components of the parametric map $\bG$:
\begin{theorem}\label{thm:int}
  Let Assumption \ref{ass:1} be satisfied for some $s>1$ and $t>0$,
  and let the interpolation points $(\chi_j)_{j\in\N_0}$ be such that
  \eqref{eq:lebesgue} holds. 
  Fix $\delta>0$ (arbitrarily small).

  Then, there exists a constant $C>0$ (depending on $\delta$, $\tau$,
  $s$, $t$, $r$) such that, for every $N\in\N$, there exist downward
  closed index sets $(\Lambda_{N,j})_{j\le N}$ such that
  $\sum_{j=1}^N|\Lambda_{N,j}|\le N$ and with the interpolated
  coefficients
  ${p_N} (\bsy) =
  (I_{\Lambda_{N,j}}\dup{u(\bsy)}{\tilde{\eta}_j})_{j\le N}$ holds
  \begin{equation}\label{eq:l2errint}
    \sup_{\bsy\in U}\normlr[\cY]{\cG(\sigma_r^s(\bsy))-\cD(p_N(\bsy))}
       \le C N^{-\min\{s-1,t\}+\delta}.
  \end{equation}
\end{theorem}
\begin{remark}
  The convergence rate is in terms of
  $N\ge \sum_{j=1}^N|\Lambda_{N,j}|$, which is an upper bound of the
  number of required evaluations of $\dup{u(\bsy)}{\tilde{\eta}_j}$ for all
  $j\in\N$ (here $u$ is as in \eqref{eq:u}).
\end{remark}

{
Before giving the proof of the theorem, we first show a variant of
Lemma \ref{lemma:ml} required for the proof.

\begin{lemma}\label{lemma:ml2}
  Let $\alpha > 1$, $\beta>0$ and assume given three
  sequences $(a_i)_{i\in\N}$, $(d_j)_{j\in\N}$,
  $(e_k)_{k\in\N} \subset (0,\infty)^\N$
  with $\sum_{i\in\N}a_ie_i<\infty$
  and $d_j\lesssim j^{-\beta}$, $e_k^{-1}\lesssim k^{1-\alpha}$, 
  for all $i$, $k\in\N$.
  Assume additionally
  that $(d_j)_{j\in\N}$, $(e_k)_{k\in\N}$ are monotonically decreasing.
  Suppose that
  there exists $C<\infty$ such that the sequence
  $(c_{i,j})_{j,i\in\N}$ satisfies 
  \begin{equation*}
\forall i\in \N:\quad \sum_{j\in\N}c_{i,j}^2d_j^{-2}\le C^2 a_i^2.
\end{equation*}

Fix $\delta>0$. Then for every $N\in\N$ exists $(m_i)_{i\in\N}\subseteq\N_0$
    monotonically decreasing such that $\sum_{i\in\N}m_i\le N$ and
  \begin{equation*}
    \sum_{i\in\N}\left(\sum_{j> m_i}c_{i,j}^2\right)^{1/2}\lesssim N^{-\min\{\alpha-1,\beta\}+\delta},
  \end{equation*}
\end{lemma}
\begin{proof}
  Fix $n\in\N$. Set
  $m_i:=\lceil(\frac{n}{i})^{(\alpha-1)/\beta}\rceil$ for $i\le n$ and
  $m_i:=0$ otherwise. For $i\le n$, since $(d_j)_{j\in\N}$ was assumed
  monotonically decreasing,
  \begin{equation*}
    \left(\sum_{j>m_i}c_{i,j}^2\right)^{1/2}
    =
    \left(\sum_{j>m_i}c_{i,j}^2d_{j}^2 d_j^{-2}\right)^{1/2}
    \le C d_{m_i} a_i
    \lesssim
    m_i^{-\beta} a_i%
  \end{equation*}
  and for $i>n$ we have
  $(\sum_{j>m_i}c_{i,j}^2)^{1/2}=(\sum_{j\ge 1}c_{i,j}^2)^{1/2}\lesssim a_i$.
  Note that for $n\in\N$ due to the monotonicity of $(e_k)_{k\in\N}$
  and because $\sum_{i\in\N}a_ie_i<\infty$
  \begin{equation*}
    \sum_{i>n}a_i = \sum_{i>n}a_ie_i e_i^{-1}\lesssim e_n^{-1} \lesssim n^{1-\alpha}.
  \end{equation*}
  Similarly,
  \begin{equation*}
    \sum_{i\le n} m_i^{-\beta}a_i = \sum_{i\le n} m_i^{-\beta}a_ie_ie_i^{-1}
    \lesssim \max_{1\le i\le n}m_i^{-\beta}e_i^{-1}\lesssim
    \max_{1\le i\le n}\Big(\frac{n}{i}\Big)^{1-\alpha}i^{1-\alpha}=n^{1-\alpha}.
  \end{equation*}
  Thus
  \begin{align*}
    \sum_{i\in\N}\left(\sum_{j> m_i}c_{i,j}^2\right)^{1/2}
    \lesssim \sum_{i\le n} m_i^{-\beta}a_i+\sum_{i>n}a_i
    \lesssim n^{-\alpha+1}.
  \end{align*}
In addition,
  \begin{equation*}
    \sum_{j\in\N}m_j\lesssim n+n^{\frac{\alpha-1}{\beta}}\int_1^{n+1} x^{-\frac{\alpha-1}{\beta}}\dd x
    \lesssim
    n+n^{\frac{\alpha-1}{\beta}}
    \begin{cases}
      1&\text{if }\frac{\alpha-1}{\beta}>1\\
      \log(n) &\text{if }\frac{\alpha-1}{\beta}=1\\
      n^{1-\frac{\alpha-1}{\beta}}&\text{if }\frac{\alpha-1}{\beta}<1
    \end{cases}
    \lesssim
    \begin{cases}
      n^{\frac{\alpha-1}{\beta}}&\text{if }\frac{\alpha-1}{\beta}>1\\
      n\log(n) &\text{if }\frac{\alpha-1}{\beta}=1\\
      n&\text{if }\frac{\alpha-1}{\beta}<1.
    \end{cases}    
  \end{equation*}
  With $M:=\sum_{j\in\N}m_j$ we get
  \begin{equation*}
    \sum_{i\in\N}\left(\sum_{j> m_i}c_{i,j}^2\right)^{1/2}\lesssim
    \begin{cases}
      M^{-\beta+\delta}&\text{if }\alpha-1\ge \beta\\
      M^{-\alpha+1+\delta}&\text{if }\alpha-1<\beta.
    \end{cases}        
  \end{equation*}
  Choosing $n(N)$ appropriately we can guarantee $M(n)\sim N$.
\end{proof}
}

\begin{proof}[Proof of Theorem \ref{thm:int}]
  { Fix $p\in (\frac{1}{s},1]$ such that $\frac{1}{p}\ge
    s-\delta/2$.  Consider Theorem \ref{thm:bpe} (with
    ``$\tau$'' being
    $\frac{3}{2}+\tau$ for the value in \eqref{eq:lebesgue}), and let
    $(a_\bsnu)_{\bsnu\in\cF}\in\ell^p(\cF)$ and
    $(e_\bsnu)_{\bsnu\in\cF}$ be as in this theorem.  Moreover, let
    $(\bsnu_i)_{i\in\N}$ be an arbitrary enumeration such that
    $(e_{\bsnu_i}^{-1})_{i\in\N}$ is monotonically decreasing.  According
    to Theorem \ref{thm:bpe} \ref{item:cprop} it holds in particular
    \begin{subequations}\label{eq:lemmaassint}    
    \begin{equation}
      \sum_{i\in\N}a_{\bsnu_i}e_{\bsnu_i}<\infty
    \end{equation}
    and $(e_\bsnu^{-1})_{\bsnu\in\cF}\in\ell^{p/(1-p)}(\cF)$.
    Due to
    $i e_{\bsnu_i}^{-p/(1-p)}\le \sum_{j\in\N}e_{\bsnu_j}^{-p/(1-p)}<\infty$ this
    implies
    \begin{equation}
      e_{\bsnu_i}^{-1}\lesssim i^{-\frac{1-p}{p}} \le i^{1-s+\frac{\delta}{2}}.
    \end{equation}

    Next fix $t'\in [0,t)$ such that $t'>t-\delta/2$. By
    Proposition \ref{prop:weightsum}
    with $\omega_{\bsnu}:=\prod_{j\in\supp\bsnu}(1+2\nu_j)\ge 1$
    we have for every $i\in\N$
    \begin{equation}
      \omega_{\bsnu_i}^{3/2+\tau} \left(\sum_{j\in\N}w_j^{-2t'} c_{\bsnu_i,j}^2\right)^{1/2}\lesssim a_{\bsnu_i}
      \lesssim i^{-s+\frac{\delta}{2}}.
    \end{equation}
    In addition, due to
    $(w_j^{t'})_{j\in\N}\in\ell^{1/(t-\delta/2)}(\N)$, by the same
    argument as above (using that $(w_j)_{j\in\N}$ was assumed
    monotonically decreasing) it holds
    \begin{equation}
      w_j^{t'}\lesssim j^{-t+\frac{\delta}{2}}.
    \end{equation}
    \end{subequations}
  }
    
    Due to \eqref{eq:lemmaassint} with $\alpha:=s-\delta/2$ and
    $\beta:=t-\delta/2$, by Lemma \ref{lemma:ml2} we can find
    $(m_i)_{i\in\N}\subset\N_0$ such that $\sum_{i\in\N}m_i\le N$ and
    \begin{equation}\label{eq:intrate}
      \sum_{i\in\N}\omega_{\bsnu_i}^{3/2+\tau}\left(\sum_{j> m_i}c_{\bsnu_i,j}^2\right)^{1/2} \le N^{-\min\{s-1,t\}+\delta}.
    \end{equation}
    Now define for $j\le N$
    \begin{equation}\label{eq:LambdaNj}
      \Lambda_{N,j}:=\set{\bsnu_i}{m_i\ge j}=\set{\bsnu_i}{i\le \max\set{r}{m_r\ge j}},
    \end{equation}
    where the equality follows by the fact that
    $(m_i)_{i\in\N}$ is monotonically decreasing
    according to Lemma \ref{lemma:ml}.
    Thus each $\Lambda_{N,j}$ is downward closed by Theorem \ref{thm:bpe}. 
    In addition,
    \begin{equation*}
      \sum_{j\in\N}|\Lambda_{N,j}|=\sum_{j\in\N}\sum_{\set{i}{m_i\ge j}}1
      =\sum_{i\in\N}\sum_{\set{j}{j\le m_i}}1=\sum_{i\in\N}m_i \le N.
    \end{equation*}

    With \eqref{eq:expansion} it holds
    \begin{equation*}
      I_{\Lambda_{N,j}}\dup{u}{\eta_j}
      =I_{\Lambda_{N,j}}\sum_{\bsnu\in\cF}c_{\bsnu,j}L_\bsnu
      =\sum_{\bsnu\in\Lambda_{N,j}}c_{\bsnu,j}L_\bsnu
      +\sum_{\bsnu\in\cF\backslash\Lambda_{N,j}}c_{\bsnu,j}I_{\Lambda_{N,j}}L_\bsnu,
    \end{equation*}
    where we used that $I_{\Lambda_{N,j}}L_{\bsnu}=L_\bsnu$ for all
    $\bsnu\in\Lambda_{N,j}$ by Lemma \ref{lemma:prop} \ref{item:int}.
    
    Thus for all $\bsy\in U$ with
    $p_N(\bsy)=(I_{\Lambda_{N,j}}\dup{u}{\eta_j})_{j=1}^N$
    \begin{align*}
      & \norm[\cY]{\cG(\sigma_r^s(\bsy))-\cD(p_N(\bsy))}
\\
      &=\normlr[\cY]{\sum_{j\in\N}\sum_{\bsnu\in\cF}c_{\bsnu,j}L_{\bsnu}(\bsy)\eta_j
        - \sum_{j\in\N}\sum_{\bsnu\in\Lambda_{N,j}}c_{\bsnu,j}L_{\bsnu}(\bsy)\eta_j
        - \sum_{j\in\N}\sum_{\bsnu\in\cF\backslash\Lambda_{N,j}}c_{\bsnu,j}I_{\Lambda_{N,j}}L_{\bsnu}(\bsy)\eta_j}\nonumber\\
      &\le
        \normlr[\cY]{\sum_{j\in\N}\sum_{\set{\bsnu_i}{m_i<j}}
        c_{\bsnu_i,j}L_{\bsnu_i}(\bsy)\eta_j}
        +\normlr[\cY]{\sum_{j\in\N}\sum_{\set{\bsnu_i}{m_i<j}}
        c_{\bsnu_i,j}(L_{\bsnu_i}(\bsy) - I_{\Lambda_{N,j}} L_{\bsnu_i}(\bsy))\eta_j }\nonumber\\
      &=
        \normlr[\cY]{\sum_{i\in\N}\sum_{j>m_i}
        c_{\bsnu,j}L_{\bsnu}(\bsy)\eta_j}
        +\normlr[\cY]{\sum_{i\in\N}\sum_{j>m_i}
        c_{\bsnu_i,j}(L_{\bsnu_i}(\bsy) - I_{\Lambda_{N,j}} L_{\bsnu_i}(\bsy))\eta_j }\nonumber\\
      &\le \Lambda_\cY
        \sum_{i\in\N} 
          \left(\sum_{j>m_i}c_{\bsnu_i,j}^2\norm[L^\infty(U)]{L_{\bsnu_i}}^2\right)^{1/2}
        + \left(\sum_{j>m_i}c_{\bsnu_i,j}^2(\norm[L^\infty(U)]{L_{\bsnu_i}}+\norm[L^\infty(U)]{I_{\Lambda_{N,j}}L_{\bsnu_i}})^2\right)^{1/2},
    \end{align*}
    where $\Lambda_\cY$ denotes again the upper frame constant in
      \eqref{eq:FrBd} for the frame $\bsPsi_\cY=(\eta_j)_{j\in\N}$,
      cp.~Rmk.~\ref{rmk:FrBd}. 
    By Lemma \ref{lemma:prop} \ref{item:lebesgue} and \eqref{eq:Lnubound} we have
    \begin{equation*}
      \norm[L^\infty(U)]{L_{\bsnu_i}}+\norm[L^\infty(U)]{I_{\Lambda_{N,j}}L_{\bsnu_i}}\le 2\omega_{\bsnu}^{3/2+\tau}.
    \end{equation*}
    Thus, using \eqref{eq:intrate} we find
    \begin{equation*}
      \sup_{\bsy\in U}\norm[\cY]{\cG(\sigma_r^s(\bsy))-\cD(p_N(\bsy))}
      \le 3 \Lambda_\cY N^{-\min\{s-1,t\}+\delta}
    \end{equation*}
    which concludes the proof.
  \end{proof}

  Theorem \ref{thm:main2} is now a direct consequence of Theorem
  \ref{thm:int}; specifically the statement of Theorem \ref{thm:main2}
  follows after introducing a scaling to map
  $\times_{j\in\N}[-rw_j^s,rw_j^s]\to U$ as in the proof of Theorem
  \ref{thm:main}.

  {
    \section{Implementation of spectral operator surrogates}
    As mentioned above, sparse-grid interpolation is deterministic,
    and the proof of Theorem \ref{thm:int} is constructive. Thus the
    polynomial surrogate in Theorem \ref{thm:int} can be explicitly
    computed. For the convenience of the reader we summarize the
    procedure, and sketch out the algorithm in the following. In
    particular, this provides a ``training algorithm'' for the
    determination of surrogates. Its practicality and performance
    remains to be investigated, which we leave for future work.

    The %
    proof of Theorem \ref{thm:int}
    proceeds in three steps: For fixed $N\in\N$
    \begin{enumerate}
    \item\label{item:LambdaNj} Determine the sets $\Lambda_{N,j}\subseteq\cF$ for $j=1,\dots,N$.
    \item For $\bsy\in [-1,1]^\N$ let
      \begin{equation}
        f_j(\bsy):=\duplr{\cG\Big(\sum_{j\in\N}y_j\psi_j\Big)}{\tilde\eta_j}
      \end{equation}
      and compute the interpolants
      \begin{equation}\label{eq:pNj}
        p_{N,j}(\bsy):=I_{\Lambda_{N,j}}[f_j](\bsy)
      \end{equation}
      using \eqref{eq:Inu} and \eqref{eq:ILambda} for all $j=1,\dots,N$.
    \item Compute the operator surrogate
      \begin{equation*}
        \cD\circ p_N\circ\cE(a) = \sum_{j=1}^N \eta_j p_{N,j}(\cE(a)).
      \end{equation*}
    \end{enumerate}
    While the interpolation polynomials $p_{N,j}$ in \eqref{eq:pNj}
    formally are functions of $\bsy\in [-1,1]^\N$, due to each
    multiindex $\bsnu\in\Lambda_{N,j}$ having finite support, in
    practice these are polynomials depending only on the finitely many
    variables
    $\set{y_j}{\exists\bsnu\in\Lambda_{N,j}~\text{s.t.}~\nu_j\neq 0}$.
    Hence they are computable.

    The critical step is the determination of the index sets
    $\Lambda_{N,j}$ in \ref{item:LambdaNj}. These sets are constructed
    in the proofs of Lemma \ref{lemma:ml2} and Theorem
    \ref{thm:int}. Simplifying a bit by ignoring logarithmic terms and
    the (positive but arbitrarily small) constant $\delta>0$, the
    procedure reads: For fixed $N\in\N$
    \begin{enumerate}
    \item Set $m_{N,j}:=(n(N)/j)^{(s-1)/t}$, $j=1,\dots,N$, where
      \begin{equation*}
        n(N):= \begin{cases}
          N^{\frac{t}{s-1}} &\text{if }s-1>t\\
          N &\text{otherwise.}
          \end{cases}
        \end{equation*}
      \item With $T>1$, $t>0$ as in Theorem \ref{thm:bpe}, define
        \begin{equation}\label{eq:Tt}
          e_\bsnu:=\prod_{j\in\supp\bsnu}\rho_j^{\nu_i}\qquad\text{where}\qquad
          \rho_j:=\max\{T,tw_j^{s(1-p)}\}.
        \end{equation}
        where an empty product equals $1$.
      \item\label{item:enumeration} Determine an enumeration $(\bsnu_i)_{i\in\N}$ of $\cF$
        such that $(e_{\bsnu_i}^{-1})_{i\in\N}$ is monotonically
        decreasing.
      \item For $j=1,\dots,N$ let (as in \eqref{eq:LambdaNj})
      \begin{equation*}
        \Lambda_{N,j} = \set{\bsnu_i}{i\le\max\set{r}{m_{N,r}\ge j}}
        \simeq \set{\bsnu_i}{i\le n(N) j^{t/(1-s)}}.
      \end{equation*}
    \end{enumerate}
 
    Again, all steps are straightforward, except for
    \ref{item:enumeration}. There exist different algorithms
    (recursive and non-recursive) capable of determining
    $\bsnu_1,\dots,\bsnu_N$ in linear complexity. We refer to
    \cite[Alg.~4.13]{MR2566594}, \cite[Alg.~2]{JZDiss} and
    \cite[Alg.~6]{2311.04172}.

    Finally, we emphasize that the constants $T>1$ and $t>0$ in
    \eqref{eq:Tt} are unknown. 
    In practice they may be set to $1$, or
    determined by trial and error. While the precise tuning of the
    constants can make a difference in the performance,
    different constructions seem to yield asymptotically similar
    results, as long as the correct behaviour with regards to the
    importance of each dimension is captured. We refer for instance to
    the numerical experiments in \cite[Sec.~5.4]{MR4113052} and \cite[Chapter 5]{JZDiss}.
  }
\section{Examples}
\label{sec:Expl}
\subsection{Diffusion equation on the torus}
\label{sec:DiffEqTorus}
Denote by $\bbT^d\simeq [0,1]^d$ the $d$-dimensional torus, $d\in\N$.
In the following all function spaces on $\bbT^d$ are assumed to be
$1$-periodic with respect to each variable.

\subsubsection{Operator $\cG$}\label{sec:torusG}
Given a nominal
coefficient $\bar{a}\in L^\infty(\bbT^d)$, a diffusion coefficient
$a\in L^\infty(\bbT^d)$, and a source $f\in H^{-1}(\bbT^d)/\R$, we
wish to find $u\in H^1(\bbT^d)$ such that
\begin{equation} \label{eq:DiffPeriodic}
 - \nabla \cdot ((\bar{a} + a )\nabla u) = f 
\text{ on } \bbT^d\qquad
\text{ and }\qquad \int_{\bbT^d}u(x)\dd x  =0
\end{equation}
in a weak sense. 
Assuming
\begin{equation}\label{eq:u_bar}
    \essinf_{x\in \bbT^d} (\bar{a}(x)+a(x)) > a_{\min},
\end{equation}
for some constant $a_{\min}>0$, it follows by the Lax-Milgram Lemma
that \eqref{eq:DiffPeriodic} has a unique solution
$u\in H^1(\bbT^d)/\R$ that we denote by $\cG(a):=u$. Thus $\cG$ is a
well-defined map from
$\set{a\in L^\infty(\bbT^d)}{\text{\eqref{eq:u_bar} holds}}\to
H^1(\bbT^d)/\R\hookrightarrow H^1(\bbT^d)$.

\subsubsection{$\cX^s$ and $\cY^t$}\label{sec:cXscYt_torus}
We first construct the usual Fourier basis on $L^2(\bbT^d)$: Set for $j\in\N$
$$ 
\xi_0 = 1,\qquad
\xi_{2j}(x)   = (2\pi)^{-1/2}\cos(2\pi j x),\qquad
\xi_{2j-1}(x) = (2\pi)^{-1/2}\sin(2\pi jx),
$$
and for $d\ge 2$ and $\bsj\in\N_0^d$
\begin{equation}\label{eq:xibsj}
\xi_\bsj(x_1,\dots,x_d):= \prod_{k=1}^d \xi_{j_k}(x_k).
\end{equation}
Then $\set{\xi_\bsj}{\bsj\in\N_0^d}$ is an ONB of
$L^2(\bbT^d)$. 
Recall that for $s\ge 0$,
holds
\begin{equation}\label{eq:HsChar}
  H^s(\bbT^d)=\setlr{u\in L^2(\bbT^d)}{\sum_{\bsj\in\N_0^d}\dup{u}{\xi_\bsj}^2\max\{1,|\bsj|\}^{2s}<\infty},
\end{equation}
where throughout we consider $H^s(\bbT^d)$
equipped with the inner product
\begin{equation}\label{eq:inpHs}
  \dup{u}{v}_{H^s}:=\sum_{\bsj\in\N_0^d}\dup{u}{\xi_\bsj}_{L^2}\dup{v}{\xi_\bsj}_{L^2}\max\{1,|\bsj|\}^{2s}.
\end{equation}

Fixing $s_0$, $t_0\ge 0$ (to be chosen later) we let
\begin{equation}\label{eq:cXscYt_torus}
  \begin{aligned}
  \cX&:=H^{s_0}(\bbT^d),\qquad\psi_\bsj:=\max\{1,|\bsj|\}^{-s_0}\xi_\bsj,\\
  \cY&:=H^{t_0}(\bbT^d),\qquad\eta_\bsj:=\max\{1,|\bsj|\}^{-t_0}\xi_\bsj,
  \end{aligned}
\end{equation}
so that $\bsPsi_\cX:=(\psi_{\bsj})_{\bsj\in\N_0^d}$,
$\bsPsi_\cY:=(\eta_{\bsj})_{\bsj\in\N_0^d}$ form ONBs of $\cX$, $\cY$
respectively. Next, introduce the weight sequence
\begin{equation*}
  w_\bsj:=\max\{1,|\bsj|\}^{-d}\qquad\bsj\in\N_0^d,
\end{equation*}
so that $(w_\bsj^{1+\eps})_{\bsj\in\N_0^d}\in\ell^1(\N_0^d)$ for any $\eps>0$ as
required in Sec.~\ref{sec:Scales}. Then, by \eqref{eq:HsChar} and the
definition of $\cX^s$ in \eqref{eq:XsYt}, it holds for $s\ge 0$
\begin{align*}
  \cX^s&=\setlr{u\in H^{s_0}(\bbT^d)}{\sum_{\bsj\in\N_0^d} \dup{u}{\psi_\bsj}_{H^{s_0}}^2w_\bsj^{-2s}<\infty}\nonumber\\
       &=\setlr{u\in L^{2}(\bbT^d)}{\sum_{\bsj\in\N_0^d} \dup{u}{\xi_\bsj}_{L^2}^2\max\{1,|\bsj|\}^{2s_0+2sd}<\infty} = H^{s_0+sd}(\bbT^d).
\end{align*}
Here we used that for $u=\sum_{\bsj\in\N_0}c_\bsj\xi_\bsj$ holds
by \eqref{eq:inpHs}
\begin{equation*}
  \dup{u}{\psi_\bsj}_{H^{s_0}}=
  \dup{u}{\xi_\bsj\max\{1,|\bsj|\}^{-s_0}}_{H^{s_0}}=
  c_\bsj\max\{1,|\bsj|\}^{s_0}=\dup{u}{\xi_\bsj}_{L^2}\max\{1,|\bsj|\}^{s_0}.
\end{equation*}
By the same argument $\cY^t=H^{t_0+td}(\bbT^d)$ for any $t\ge 0$.

\subsubsection{Coefficient-to-solution surrogate rates}
\label{sec:OpSurrRates}

We now give a convergence result for the approximation of the solution
operator $\cG$ (corresponding to the PDE \eqref{eq:DiffPeriodic}) on a
Sobolev ball. The encoder $\cE$ and decoder $\cD$ are as in
\eqref{eq:ED}, w.r.t.\ the spaces and ONBs in \eqref{eq:cXscYt_torus},
which depend on the constants $s_0$, $t_0\ge 0$ that are still at our
disposal. The parameter $s_0$ controls the regularity of the input
space and thus determines the encoder $\cE$. It will have to be chosen
suitably in order to achieve possibly fast convergence. On the other
hand, $t_0$ controls the regularity of the output space. It may be
chosen freely and determines the norm in which the error is
measured---smaller $t_0$ amounts to a weaker norm in the output space
and thus yields larger convergence rates.

\begin{proposition}\label{prop:torus}
  Assume $f\in C^\infty(\bbT^d)/\R$.  Let $\alpha>\frac{3d}{2}$,
  $r>0$, $a_{\min}>0$ and let $\delta>0$ (arbitrarily small). Suppose
  that $\bar a+a$ satisfies \eqref{eq:u_bar} for all
  $a\in B_r(H^{\alpha}(\bbT^d))$.

  Then for every $t_0\in [0,1]$ there exists a
  constant $C>0$ and for all $N\in\N$ there exists a ReLU NN
  $\tilde\bG_N$ of size $O(N)$ such that
  \begin{subequations}\label{eq:toshow}
  \begin{equation}
    \sup_{a\in B_r(H^\alpha(\bbT^d))}\norm[H^{t_0}(\bbT^d)]{\cG(a)-\cD\circ\tilde\bG_N\circ\cE(a)}
    \le C N^{-R+\delta}
  \end{equation}
  where
  \begin{equation}
    R = \begin{cases}
      \frac{\alpha}{d}-\frac{3}{2} &\text{if }\alpha\in (\frac{3d}{2},\frac{3d}{2}+1-t_0]\\
      \frac{\alpha+1-t_0}{2d}-\frac{3}{4} &\text{if }\alpha > \frac{3d}{2}+1-t_0,
    \end{cases}
  \end{equation}
\end{subequations}
and where $\cE$, $\cD$ are as in \eqref{eq:ED} with the spaces/ONBs in \eqref{eq:cXscYt_torus} with $t_0$ from above and
\begin{equation*}
  s_0 =
  \begin{cases}
        \frac{d}{2}+\frac{\delta}{2} &\text{if }\alpha\in (\frac{3d}{2},\frac{3d}{2}+1-t_0]\\
        \frac{\alpha+t_0-\frac{d}{2}-1}{2} &\text{if }\alpha > \frac{3d}{2}+1-t_0.
        \end{cases}
    \end{equation*}
\end{proposition}
\begin{proof}
  {\bf Step 1.} We verify Assumption \ref{ass:1}.
  By
  classical elliptic regularity (Schauder estimates, for
  second order divergence-form linear elliptic equations, also with
  complex-valued coefficients,
  e.g., \cite[pg. 625]{ADNI}, \cite[Sec. 2]{ADNII}) %
  it holds for
  all $\gamma>0$
\begin{equation*}
 \cG:\set{a\in C^{\gamma}(\bbT^d)}{\bar a + a\text{ satisfies \eqref{eq:u_bar}}}
  \to C^{1+\gamma}(\bbT^d)\hookrightarrow H^{1+\gamma}(\bbT^d),
\end{equation*}
and that $\cG(a)$ is bounded on bounded subsets of
$\set{a\in C^{\gamma}(\bbT^d)}{\bar a + a\text{ satisfies
    \eqref{eq:u_bar}}}$.

Thus, if
\begin{equation}\label{eq:const1}
  s_0>\frac{d}{2},
\end{equation}
{using that for any
$\gamma\in [0,s_0-\frac{d}{2})$ holds the Sobolev embedding}
$H^{s_0}\hookrightarrow C^{\gamma}$, we find
\begin{equation*}
  \cG:\set{a\in H^{s_0}}{\bar a+a\text{ satisfies \eqref{eq:u_bar}}}\to C^{1+\gamma}\qquad \forall \gamma\in \Big[0,s_0-\frac{d}{2}\Big).
\end{equation*}
We require \eqref{eq:const1} to ensure
$H^{s_0}(\bbT^d)\hookrightarrow L^\infty(\bbT^d)$, which is
necessary in order for $\cG$ to be well-defined, see Section
\ref{sec:torusG}.
In addition,
\begin{equation*}
  C^{1+\gamma}(\bbT^d)\hookrightarrow H^{1+\gamma}(\bbT^d) =\cY^t
\end{equation*}
with $t\ge 0$ such that $t_0+td=1+\gamma$, i.e.\
$t=\frac{1+\gamma-t_0}{d}$ ($t\ge 0$ holds since by assumption
$t_0\le 1 \le 1+\gamma$). With $\cX=H^{s_0}(\bbT^d)$ this shows
\begin{equation*}
  \cG:\set{a\in \cX}{\bar a+a\text{ satisfies \eqref{eq:u_bar}}}\to \cY^t\qquad
  \forall t\in \Big[0,\frac{1+s_0-\frac{d}{2}-t_0}{d}\Big)
\end{equation*}
and for fixed $t$ the map is bounded on bounded subsets of
$\set{a\in \cX}{\bar a+a\text{ satisfies \eqref{eq:u_bar}}}$.

Next, if $s>1$, $\tilde C_r^s(\cX)$ (which is equal to $C_r^s(\cX)$ by
Rmk.~\ref{rmk:riesz}) is in particular a bounded subset of $\cX\hookrightarrow L^\infty(\bbT^d)$
(cp.~Rmk.~\ref{rmk:Cs_in_Xs_prime}). 
Hence, for example by
\cite[Proposition 1.2.33 and Example 1.2.38]{JZDiss} 
there exists an open complex set $O_\bbC\subset\cX_\bbC$ containing
$\tilde C_r^s(\cX)$ such that due to $t_0\in [0,1]$,
\begin{equation}
\cG:O_\bbC\to H^1(\bbT^d,\C)\hookrightarrow\cY=H^{t_0}(\bbT^d)
\end{equation}
is holomorphic. Furthermore, it follows from the a-priori estimate
\cite[Theorem 9.3]{ADNII}, which is also valid for
\eqref{eq:DiffPeriodic} with periodic boundary conditions, by
combining the $\bbT^d$-periodicity of solutions w.r. to ${\rm Re}x$
with the ``hemisphere'' a-priori bounds in \cite[Theorem 9.2]{ADNII}
on each face of $\bbT^d$, that

\begin{equation}
  \cG:O_\bbC\to \cY^t\qquad\forall t\in \Big[0,\frac{1+s_0-\frac{d}{2}-t_0}{d}\Big)
\end{equation}
is bounded. 

{\bf Step 2.} We conclude the proof. 
According to
Cor.~\ref{cor:ball}, for $s>1$ and
$t\in (0,\frac{1+s_0-\frac{d}{2}-t_0}{d})$ we have with
$\cX^s=H^{s_0+sd}(\bbT^d)$, $\cY=H^{t_0}(\bbT^d)$ and for all $\delta>0$
(arbitrarily small)
\begin{equation*}
  \sup_{a\in B_r(H^{s_0+sd}(\bbT^d))}\norm[H^{t_0}]{\cG(a)-\cD(\tilde\bG_N(\cE(a)))}\le C N^{-\min\{s-1,t\}+\delta}.
\end{equation*}
Substituting $\alpha=s_0+sd$, i.e.\ $s=\frac{\alpha-s_0}{d}$,
and taking the maximal $t$ this reads
\begin{equation*}
  \sup_{a\in B_r(H^{\alpha}(\bbT^d))}\norm[H^{t_0}]{\cG(a)-\cD(\tilde\bG_N(\cE(a)))}\le C N^{-\min\{\frac{\alpha-s_0}{d}-1,\frac{1+s_0-\frac{d}{2}-t_0}{d}\}+\delta}.
\end{equation*}
The constraint $s>1$ implies the constraint $\alpha>d+s_0$ on
$\alpha$. We are still free to choose $s_0>\frac{d}{2}$, and wish to
do so in order to maximize the resulting convergence rate.
Solving
\begin{equation*}
  \frac{\alpha-s_0}{d}-1 = \frac{1+s_0-\frac{d}{2}-t_0}{d}
\end{equation*}
for $s_0$ we have
\begin{equation}\label{eq:s0}
  s_0 = \frac{\alpha+t_0-\frac{d}{2}-1}{2}.
\end{equation}
The constraint $s_0>\frac{d}{2}$ implies the constraint
$\alpha>\frac{3d}{2}+1-t_0$.

We therefore now distinguish between two cases. First, if
$\alpha\in (\frac{3d}{2},\frac{3d}{2}+1-t_0]$, then we let
$s_0:=\frac{d}{2}+\eps$ for some small $\eps>0$ ($\eps>0$ small enough
implies in particular that $\alpha>d+s_0$). In this case we obtain, up
to some arbitrarily small $\delta>0$, the convergence rate
\begin{equation*}
  \min\Big\{\frac{\alpha-\frac{d}{2}}{d}-1,\frac{1+\frac{d}{2}-\frac{d}{2}-t_0}{d}\Big\}
  =\min\Big\{\frac{\alpha}{d}-\frac{3}{2},\frac{1-t_0}{d}\Big\}=\frac{\alpha}{d}-\frac{3}{2},
\end{equation*}
where we used that $\alpha\le \frac{3d}{2}+1-t_0$ for the last equality.

Next assume $\alpha>\frac{3d}{2}+1-t_0$
and define $s_0$ as in \eqref{eq:s0}.
The constraint $\alpha>d+s_0$ is then equivalent to
\begin{equation*}
  \alpha>d+\frac{\alpha+t_0-\frac{d}{2}-1}{2}
  \qquad\Leftrightarrow\qquad
  \alpha>\frac{3d}{2}+t_0-1,
\end{equation*}
which already holds since $\alpha>\frac{3d}{2}+1-t_0
\ge \frac{3d}{2}+t_0-1$ for all $t_0\in [0,1]$.
The convergence rate amounts in this case to
\begin{equation*}
  \frac{\alpha-s_0}{d}-1 = \frac{\alpha+1-t_0}{2d}-\frac{3}{4}.
\end{equation*}
This shows \eqref{eq:toshow}.
\end{proof}

The above proposition is based on Cor.~\ref{cor:ball}. Applying
instead Thm.~\ref{thm:main}, one obtains for example that
for all $s>1$, $s_0>\frac{d}{2}$, $t_0\in [0,1]$,
with $\cX=H^{s_0}(\bbT^d)$ %
and for $\delta>0$ fixed but arbitrarily small
(cp.~Step 2 in the proof of Prop.~\ref{prop:torus})
\begin{equation*}
  \sup_{a\in C_r^s(\cX)}\norm[H^{t_0}(\bbT^d)]{\cG(a)-\cD(\tilde\bG_N(\cE(a)))}\le C N^{-\min\{s-1,\frac{1+s_0-\frac{d}{2}-t_0}{d}\}+\delta}.
\end{equation*}
Similarly, Thm.~\ref{thm:main_riesz} gives an improved $L^2$-type
error estimate, and Thm.~\ref{thm:main2} gives a convergence
{rate bound for spectral} operator surrogates.

{
\subsection{Diffusion equation on a polygonal domain}
Denote in the following by $D \subset \bbT^2\simeq [0,1]^2$
a convex polygonal domain with positive distance
from the boundary $\partial [0,1]^2$ of the unit cube.

\subsubsection{Operator $\cG$}
Similar to Section \ref{sec:torusG}, given a nominal
coefficient $\bar{a}\in L^\infty(\bbT^2)$, a diffusion coefficient
$a\in L^\infty(\bbT^2)$, and a source $f\in H^{-1}(D)$, we
wish to find $u\in H^1(D)$ such that 
\begin{equation} \label{eq:DiffPeriodic2}
 - \nabla \cdot ((\bar{a} + a )\nabla u) = f 
 \text{ on } D\qquad\text{and}\qquad
 u|_{\partial D}\equiv 0
\end{equation}
in a weak sense. Assuming
\begin{equation}\label{eq:u_bar2}
    \essinf_{x\in D} (\bar{a}(x)+a(x)) > a_{\min},
\end{equation}
for some constant $a_{\min}>0$, it follows by the Lax-Milgram Lemma
that \eqref{eq:DiffPeriodic} has a unique solution
$u\in H_0^1(D)$ that we denote by $\cG(a):=u$. Thus $\cG$ is a
well-defined map from
$\set{a\in L^\infty(\bbT^2)}{\text{\eqref{eq:u_bar2} holds}}\to
H_0^1(D)\hookrightarrow H^1(D)$.

\subsubsection{$\cX^s$ and $\cY^t$}\label{sec:cXcYpolygon}
For the input space, we use the same representation as in Section \ref{sec:cXscYt_torus}. 
That is, for some $s_0\ge 0$ (to be chosen later) we let 
with \eqref{eq:xibsj} for all $\bsj\in\N_0^2$
\begin{equation}\label{eq:cXpolygon}
  \cX:=H^{s_0}(\bbT^2),\qquad\psi_\bsj:=\max\{1,|\bsj|\}^{-s_0}\xi_\bsj,
\end{equation}
so that $\bsPsi_\cX:=(\psi_{\bsj})_{\bsj\in\N_0^2}$
forms an ONB of $\cX$. 
With the weight sequence
\begin{equation*}
  w_\bsj^{\cX}:=\max\{1,|\bsj|\}^{-d}\qquad\bsj\in\N_0^2,
\end{equation*}
this then amounts to
\begin{equation*}
  \cX^{s} = H^{s_0+2s}(\bbT^2)
\end{equation*}
as explained in Section \ref{sec:cXscYt_torus}.

For the output space, we cannot resort to Fourier representations of
the solution $u$ of \eqref{eq:DiffPeriodic2}, since the underlying
domain is not the torus. Therefore we employ a frame
representation. Specifically, the authors in
\cite{DavydovStevenson2006} provide an explicit construction of
functions $\hat\eta_{n,j}$, $j\in J_n$, $n\in\N$, 
for certain index sets $J_n\subseteq\N$; 
their cardinality is bounded according to 
\begin{equation}\label{eq:9n}
  |J_n|\lesssim 9^n\qquad\forall n\in\N,
\end{equation}
as a consequence of the particular basis construction in
\cite{DavydovStevenson2006}, which is based on the successive
subdivision of quadrilaterals into nine sub-quadrilaterals (see
\cite[Section 2]{DavydovStevenson2006}). For all $t\in (0,3/2)$,
it is shown in \cite{DavydovStevenson2006} that the sequence
$(3^{-tn}\hat\eta_{n,j})_{n,j}$ constitutes a Riesz basis of
$H^{1+t}(D)$.  In particular, every element of $H^{1+t}$ can be
expanded (uniquely) in terms of the $\hat\eta_{n,j}$ and
\begin{equation}\label{eq:H1t}
  \normlr[H^{1+t}(D)]{\sum_{n,j}c_{n,j}\hat\eta_{n,j}}^2\simeq \sum_{n,j}3^{2tn}c_{n,j}^2.
\end{equation}
For details see \cite{DavydovStevenson2006}, and in particular
the top of page 389 in that reference.

Introducing the weight sequence
\begin{equation*}
  w_{n,j}^{\cY}:=3^{-2n}\qquad j\in J_n,~n\in\N,
\end{equation*}
we note that for any $\eps>0$ due to \eqref{eq:9n}
\begin{equation*}
  \sum_{n,j}(w_{n,j}^{\cY})^{1+\eps}
  = \sum_{n\in\N} \sum_{j\in J_n} 3^{-2n(1+\eps)}
  \le \sum_{n\in\N} 3^{-2n(1+\eps)+2n}
  =\sum_{n\in\N} 3^{-2\eps n}
  <\infty
\end{equation*}
as required.

We fix in the following $\delta>0$ and let
\begin{equation}\label{eq:basispolygon}
  \eta_{n,j}:=\hat\eta_{n,j} (w_{n,j}^{\cY})^{\delta/2}
\end{equation}
and
\begin{equation*}
  \cY = \setlr{\sum_{n,j}c_{n,j}\eta_{n,j}}{{\normlr[\cY]{\sum_{n,j}c_{n,j}\eta_{n,j}}^2}:=\sum_{n,j}c_{n,j}^2<\infty}.
\end{equation*}
Due to \eqref{eq:H1t} it then holds
\begin{align}\label{eq:cYpolygon}
  \cY &= \setlr{\sum_{n,j}c_{n,j} (w_{n,j}^{\cY})^{\delta/2}\hat\eta_{n,j}}{\sum_{n,j}c_{n,j}^2<\infty}\nonumber\\
  &= \setlr{\sum_{n,j}d_{n,j} \hat\eta_{n,j}}{\sum_{n,j}d_{n,j}^2(w_{n,j}^{\cY})^{-\delta}<\infty}
  = H^{1+\delta}(D).
\end{align}
Similarly, for $t\in (0,3/4-\delta/2)$ we have due to \eqref{eq:H1t}
\begin{align}\label{eq:cYt}
  \cY^t&=\setlr{\sum_{n,j}c_{n,j}\eta_{n,j}}{\sum_{n,j}c_{n,j}^2 (w_{n,j}^{\cY})^{-2t}<\infty}\nonumber\\
  &=\setlr{\sum_{n,j}d_{n,j}\hat\eta_{n,j}}{\sum_{n,j}c_{n,j}^2 (w_{n,j}^{\cY})^{-2t-\delta}<\infty}
  = H^{1+2t+\delta}(D).
\end{align}

\subsubsection{Coefficient-to-solution surrogate rates}

Analogous to Proposition \ref{prop:torus}, we 
now discuss a convergence result for the approximation of the solution
operator $\cG$ (corresponding to the PDE \eqref{eq:DiffPeriodic2} on
the convex polygonal domain $D\subseteq [0,1]^2$), for all diffusion
coefficients in a Sobolev ball. The encoder $\cE$ is as in
\eqref{eq:ED}, w.r.t.\ the space and orthonormal
bases in \eqref{eq:cXpolygon}. Here $s_0\ge 0$ is still at our disposal.
The decoder $\cD$ is also as in \eqref{eq:ED}, with respect to the
Riesz basis $(\eta_{n,j})_{n,j}$ of $\cY$.

The next theorem gives essentially the same convergence rates as
Proposition \ref{prop:torus}, with two restrictions: 
\begin{itemize}
\item[(i)] 
Since $D$
is (convex) polygonal, the solution $\cG(a)=u$ of
\eqref{eq:DiffPeriodic2} belongs in general at best to $H^2(D)$, so
unlike for the PDE in Section \ref{sec:DiffEqTorus} posed on the
torus, we cannot get arbitrarily high algebraic operator emulation
rates. Therefore we assume $\alpha\le 5$ in the following.
\item[(ii)] 
Since the Riesz basis from
Section \ref{sec:cXcYpolygon} is \emph{not} stable in $H^{t_0}$ for
$t_0\in [0,1]$, we only measure the error in $\cY\hookrightarrow H^1(D)$,
cp.~\eqref{eq:cYpolygon}.
\end{itemize}

\begin{proposition}\label{prop:polygon}
  Assume $f\in C^\infty(D)$. Let $\alpha\in (3,5]$,
  $r>0$, $a_{\min}>0$ and let $\delta>0$ (arbitrarily small). Suppose
  that $\bar a+a$ satisfies \eqref{eq:u_bar} for all
  $a\in B_r(H^{\alpha}(D))$.

  Then there exists a constant $C>0$ and for all $N\in\N$ there exists
  a ReLU NN $\tilde\bG_N$ of size $O(N)$ such that
  \begin{subequations}\label{eq:toshowpolygon}
  \begin{equation}
    \sup_{a\in B_r(H^\alpha(\bbT^2))}\norm[H^{1}(D)]{\cG(a)-\cD\circ\tilde\bG_N\circ\cE(a)}
    \le C N^{-R+\delta}
  \end{equation}
  where the expression rate $R$ is given by
  \begin{equation}
    R = \frac{\alpha-3}{4},
  \end{equation}
\end{subequations}
and the encoder $\cE$ is as in \eqref{eq:ED} and \eqref{eq:cXpolygon} with
$s_0 =\frac{\alpha-1}{2}$, 
and the decoder $\cD$ is as in \eqref{eq:ED} and \eqref{eq:basispolygon}.
\end{proposition}
\begin{proof}
  {\bf Step 1.} We verify Assumption \ref{ass:1}. The argument
  is analogous to Step 1 of the proof of Prop.~\ref{prop:torus}.

  By
    {\cite[Lemma 5.2]{TSGU2013} (see also \cite[Theorem  9.1.12]{hackbusch2010Elldifequ})
   for all $\bar\gamma, \gamma\in (0,1]$ such that $\bar\gamma>\gamma$}
\begin{equation*}
 \cG:\set{a\in C^{\bar\gamma}(\bbT^2)}{\bar a + a\text{ satisfies \eqref{eq:u_bar2}}}
  \to H^{1+\gamma}(\bbT^2),
\end{equation*}
and $\cG(a)$ is bounded on bounded subsets of
$\set{a\in C^{\bar\gamma}(\bbT^2)}{\bar a + a\text{ satisfies
    \eqref{eq:u_bar2}}}$.

Thus, if
\begin{equation}\label{eq:const12}
  s_0\in (1,2],
\end{equation}
  using that for any
$\bar\gamma\in [0,s_0-1)$ holds
$H^{s_0}(\bbT^2)\hookrightarrow C^{\bar\gamma}(\bbT^2)$, we find
\begin{equation*}
  \cG:\set{a\in H^{s_0}(\bbT^2)}{\bar a+a\text{ satisfies \eqref{eq:u_bar}}}\to H^{1+\gamma}(D)\qquad \forall \gamma\in [0,s_0-1).
\end{equation*}
In addition, for $\gamma\in(\delta,1]$ by \eqref{eq:cYt}
\begin{equation*}
  H^{1+\gamma}(D) =\cY^{t-\delta/2}
\end{equation*}
with $t=\gamma/2$. With $\cX=H^{s_0}(\bbT^2)$ this shows for $s_0\in (1,2]$
\begin{equation*}
  \cG:\set{a\in \cX}{\bar a+a\text{ satisfies \eqref{eq:u_bar2}}}\to \cY^{t-\delta/2}\qquad
  \forall t\in \Big (\frac{\delta}{2},\frac{s_0-1}{2}\Big)
\end{equation*}
and for fixed $t$ the map is bounded on bounded subsets of
$\set{a\in \cX}{\bar a+a\text{ satisfies \eqref{eq:u_bar2}}}$.

Next, if $s>1$, $\tilde C_r^s(\cX)$ (which is equal to $C_r^s(\cX)$ by
Rmk.~\ref{rmk:riesz}) is in particular a bounded subset of
$\cX\hookrightarrow L^\infty(\bbT^2)$
(cp.~Rmk.~\ref{rmk:Cs_in_Xs_prime}).  Hence by
\cite[Proposition 1.2.33 and Example 1.2.38]{JZDiss} there exists an
open complex set $O_\bbC\subset\cX_\bbC$ containing
$\tilde C_r^s(\cX)$ such that
\begin{equation*}
\cG:O_\bbC\to H^1(D,\C)
\end{equation*}
is holomorphic. Furthermore, it follows from the a-priori 
estimates in \cite[Chap.~III]{ADN1}
that
\begin{equation}\label{eq:regularityt2}
  \cG:O_\bbC\to \cY^{t-\delta/2}\qquad\forall t\in \Big(\frac{\delta}{2},\frac{s_0-1}{2}\Big)
\end{equation}
is bounded. 

{\bf Step 2.} We conclude the proof. According to Cor.~\ref{cor:ball},
\eqref{eq:cYpolygon} and \eqref{eq:regularityt2}, for $s>1$ and
$t\in (\frac{\delta}{2},\frac{s_0-1}{2})$ we have with
$\cX^s=H^{s_0+2s}(\bbT^2)$
\begin{align*}
  \sup_{a\in B_r(H^{s_0+2s}(\bbT^2))}\norm[H^{1}(D)]{\cG(a)-\cD(\tilde\bG_N(\cE(a)))}&\le  
                                                                                    \sup_{a\in B_r(H^{s_0+2s}(\bbT^2))}\norm[\cY]{\cG(a)-\cD(\tilde\bG_N(\cE(a)))}\nonumber\\
                                                                                     &\le C N^{-\min\{s-1,t-\delta/2\}+\delta/2}\nonumber\\
                                                                                     &\le CN^{-\min\{s-1,t\}+\delta}.
\end{align*}
Substituting $\alpha=s_0+2s$, i.e.\ $s=\frac{\alpha-s_0}{2}$,
and taking the maximal $t$ in \eqref{eq:regularityt2} this reads
\begin{equation*}
  \sup_{a\in B_r(H^{\alpha}(\bbT^2))}\norm[H^{1}]{\cG(a)-\cD(\tilde\bG_N(\cE(a)))}\le C N^{-\min\{\frac{\alpha-s_0}{2}-1,\frac{s_0-1}{2}\}+\delta}.
\end{equation*}
We are still free to choose $s_0\in (1,2]$, and wish to do so in
order to maximize the resulting convergence rate. Solving
\begin{equation*}
  \frac{\alpha-s_0}{2}-1 = \frac{s_0-1}{2}
\end{equation*}
for $s_0$ we have
\begin{equation*}
  s_0 = \frac{\alpha-1}{2}.
\end{equation*}
The constraint \eqref{eq:const12} on $s_0$ implies the constraint
$\alpha\in (3,5]$. The constraint $s>1$ implies the constraint
$\alpha>s_0+2$ which is automatically satisfied with this
choice of $s_0$. The convergence rate then amounts to
\begin{equation*}
  \frac{\alpha-s_0}{2}-1 = \frac{\alpha-3}{4}.
\end{equation*}
This shows \eqref{eq:toshowpolygon}.
\end{proof}

}
\section{Concluding Remarks and further developments}
\label{sec:Concl}
We established expression rate bounds for %
finite-parametric approximations to
nonlinear, holomorphic maps between scales of infinite-dimensional,
separable function spaces endowed with suitable stable, affine
representation systems such as frames. Our approximations are based
on combining a linear input encoder with suitable, finite-parametric
surrogates $\{\tilde{\bG}_N\}_N$ of a countably-parametric map
transforming coefficient sequences from the input encoder into
corresponding sequences for the output decoder.

While being of independent, mathematical interest, 
the present results open a perspective of `refactoring' 
key parts of widely used scientific computing methods.
We mention only Schur-complement (or Dirichlet-to-Neumann) maps 
for elliptic PDEs with variable coefficients which constitute,
in discretized form, a key component in many algorithms of
scientific computation. 

A further, broad range of applications for the considered operator 
surrogates is 
\emph{efficient numerical realization of domain-to-solution maps} 
for elliptic and parabolic PDEs. 
Upon pullback onto one common, canonical 
reference domain, physical domain shapes are encoded in variable
coefficients of the transformed PDE, and the domain-to-solution
map is equivalent to the coefficient-to-solution map.
Such maps feature the holomorphy required 
for the presently developed theory 
(e.g. 
\cite{CSZ18_2319} for Navier-Stokes equations,
\cite{HS21_2779} for nonlocal (boundary) integrodifferential operators,
\cite{JSZ17_2339} for time-harmonic Maxwell equations
). 
We mention \cite{RegaPagaAQ22} for a recent application to
deep NNs in computational physiology.
 
The main results, Theorems~\ref{thm:main} and \ref{thm:main2}, 
considered in detail the emulation of holomorphic maps $\cG$
by either {ReLU activated} NNs 
or 
by novel generalized polynomial chaos operator surrogates. 
The latter class of surrogate operators allows, in
particular, for \emph{efficient deterministic construction 
     w.r.\ to the number of the encoded input parameters}.

{As observed in \cite[Prop.~3.7]{SchwabZechAA}, 
the presently proved approximation rates by strict ReLU NNs
can also be expected for 
other neural architectures: 
non-ReLU activations satisfying the `usual' assumptions will also suffice.
As shown in \cite{stanojevic2022exact} 
strict ReLU operator expression rate bounds as shown herein
will imply the same rates for so-called spiking neural networks,
which are prototypical neuromorphic computing models.
}

The
presently developed technical tools also accomodate other
approximation architectures for the high-dimensional, 
parametric surrogate map
$\tilde{\bG}_N$ in \eqref{eq:GenForm}, e.g.  tensor-networks or
multipole operators (e.g. \cite{jin2022mionet}).

While the present results are limited to the case of bounded parameter
ranges in the basis representations of admissible input data from the spaces $\cX^s$, 
expression rates for inputs subject to a Gaussian measure on the input spaces $\cX^s$ 
will 
require admitting unbounded parameter ranges of encoded input data. 
{
Here, similar results are conceivable, but will require 
\emph{ReLU DNN emulations of Wiener polynomial chaos expansions} 
as in \cite{SZ21_982}, \cite{DNSZ23_2957}.
}

{
The surrogates $\tilde{\cG}_N$ in \eqref{eq:GenForm} were based on
\emph{linear en- and decoders}. Significant quantitative improvements
are achieveable by \emph{nonlinear encoding and decoding}. 
For example, transformer-based emulators $\cE$ 
as proposed e.g. in \cite{li2022transformer,cao2021choose} 
or manifold-decoders $\cD$ such as NOMAD in \cite{seidman2022nomad}.
}
 
Our analysis exploited the quantified holomorphy of the
function space map $\calG$ (or its countably-parametric version $\bG$)
in an essential way; while at first
sight, this may seem restrictive, in recent years large classes of
maps of engineering interest have been identified which admit this
property.  We only mention \cite{CSZ18_2319} for the stationary
Navier-Stokes equations, \cite{JSZ17_2339} for time-harmonic Maxwell
equations and \cite{HS21_2779} for shape to boundary integral operator
maps. Both, generalization error bounds and the work bounds do not
incur the curse of dimensionality, which enters in straightforward
application of classical approximation results.

The discussed gpc surrogate operator constructions assumed
availability of noise-{less} evaluations of
$\dup{\cG(a)}{\tilde{\eta}_j}$ in at most $N$ pairs of
input-output ``snapshots'' in $C_r^s(\cX)\times \tilde{\bsPsi}_\cY$. 
Accounting for effects of
``noisy'' evaluations of these functionals, 
e.g.\ through numerical discretizations, 
will be considered elsewhere.

\begin{acknowledgement}
  This work was completed while JZ visited MIT in April and May 2022,
  and while the authors attended the Erwin Schr\"odinger Institute,
  Vienna, Austria, during the ``ESI Thematic Period on Uncertainty
  Quantification'' in May and June 2022.  Excellent working conditions
  at these institutions are gratefully acknowledged.
\end{acknowledgement}

\bibliographystyle{abbrv}
\bibliography{mybibfile}
\end{document}